\documentclass[10pt, oneside, a4paper]{article}
\usepackage{geometry}
\usepackage[english]{babel}
\usepackage[T1]{fontenc}
\usepackage[latin1]{inputenc}

\usepackage{lmodern}
\usepackage{amssymb,amsmath,amscd} 
\usepackage{mathrsfs}
\usepackage{stmaryrd}

\usepackage{ntheorem}

{\theoremstyle{plain}
\theoremseparator{ --}
\newtheorem{theo}{Theorem}[section]}

{\theoremstyle{plain}
\theoremseparator{ --}
\newtheorem{coro}[theo]{Corollary}}

{\theoremseparator{ --}
\newtheorem{lemm}[theo]{Lemma}}

{\theoremseparator{ --}
\newtheorem{prop}[theo]{Proposition}}

{\theoremseparator{ --}
\newtheorem{conj}[theo]{Conjecture}}

{\theoremseparator{ --}
}

\newtheorem{defi}[theo]{Definition}

{\theoremstyle{plain}
\theorembodyfont{\normalfont}
\newtheorem{exple}[theo]{Example}}

{\theoremstyle{plain}
\theorembodyfont{\normalfont}
\newtheorem{rema}[theo]{Remark}}

{\theoremheaderfont{\itshape}
\theorembodyfont{\normalfont}
\theoremseparator{: }
\newtheorem*{proof}{Proof}}

\newcommand{\ProofEnd}{\hfill $\Box$}

\usepackage{graphicx}
\usepackage{tabularx}
\usepackage[table]{xcolor}	
\usepackage{array}
\usepackage{enumerate}

\usepackage{todonotes}

\usepackage{url}

\DeclareMathOperator{\Reg}{Reg}
\DeclareMathOperator{\ord}{ord}
\DeclareMathOperator{\Gal}{Gal}
\DeclareMathOperator{\NS}{NS}
\DeclareMathOperator{\Pic}{Pic}

\DeclareMathOperator{\Br}{Br}
\DeclareMathOperator{\DEGRE}{deg }
\renewcommand{\deg}{\DEGRE}

\newcommand{\R}{\ensuremath{\mathbb{R}}}
\newcommand{\Q}{\ensuremath{\mathbb{Q}}}
\newcommand{\C}{\ensuremath{\mathbb{C}}}
\newcommand{\Z}{\ensuremath{\mathbb{Z}}}
\renewcommand{\P}{\ensuremath{\mathbb{P}}}
\newcommand{\F}{\ensuremath{\mathbb{F}}}
\renewcommand{\H}{\ensuremath{\mathrm{H}}}

\newcommand{\e}{\ensuremath{\mathrm{e}}}
\newcommand{\dd}{\ensuremath{\mathrm{d}}}

\newcommand{\Scal}{\mathcal{S}}

\newcommand{\fr}{\mathrm{Fr}}

\newcommand{\Qbar}{\ensuremath{\bar{\mathbb{Q}}}}

\renewcommand{\bar}[1]{\ensuremath{\overline{#1}}}

\newcommand{\hhat}[1]{\ensuremath{\widehat{#1}}}

\newcommand{\nequiv}{\not\equiv}

\newcommand{\lint}{\llbracket}
\newcommand{\rint}{\rrbracket}
\renewcommand{\tt}{\mathbf{t}}

\renewcommand{\epsilon}{\varepsilon}

 \newcommand{\ind}{1\mkern-4.1mu\mathrm{l}}

\newcommand{\ie}{\textit{i.e.}{}}

\input cyracc.def
\DeclareFontFamily{U}{russian}{}
\DeclareFontShape{U}{russian}{m}{n}
	{ <5><6> wncyr5
	<7><8><9> wncyr7
	<10><10.95><12><14.4><17.28><20.74><24.88> wncyr10 }{}
\DeclareSymbolFont{Russian}{U}{russian}{m}{n}
\DeclareSymbolFontAlphabet{\mathcyr}{Russian}
\makeatletter
\let\@math@cyr\mathcyr
\renewcommand{\mathcyr}[1]{\@math@cyr{\cyracc #1}}
\makeatother

\newcommand{\ferm}{\ensuremath{\mathcal{F}}}
\newcommand{\sbgp}[2]{{\langle#1\rangle_{#2}}}

\newcommand{\spval}[1]{P^\ast(#1)}

\newcommand{\ee}{\mathbf{e}}

\definecolor{myblue}{rgb}{0.2,0.6,1}
\definecolor{myblue2}{rgb}{0,0.4,0.8}
\definecolor{noir}{rgb}{0,0,0}
\definecolor{boxescol}{rgb}{0.25,0.25,0.25}

\usepackage{bm}

\newcommand{\intent}[1]{\llbracket #1\rrbracket}
\newcommand{\partfrac}[1]{\left\{#1\right\}}
\newcommand{\partint}[1]{\left\lfloor#1\right\rfloor}

\newcommand{\tors}{_{\mathrm{tors}}}

\newcommand{\Ja}{\mathbf{J}}

\newcommand{\bA}{\mathbf{A}}
\newcommand{\ba}{\mathbf{a}}

\newcommand{\bB}{\mathbf{B}}
\newcommand{\bb}{\mathbf{b}}

\newcommand{\bc}{\mathbf{c}}

\newcommand{\cchi}{\hm{\chi}}
\newcommand{\zzeta}{\hm{\zeta}}
\newcommand{\xx}{\hm{x}}

\newcommand{\norm}{\mathbf{N}}
\newcommand{\gP}{\mathfrak{P}}
\newcommand{\gp}{\mathfrak{p}}

\usepackage[colorlinks=true,citecolor=myblue2, linkcolor=noir,urlcolor=myblue2, pagebackref]{hyperref}

\title{A  Brauer-Siegel theorem for Fermat surfaces over finite fields}
\author{Richard {Griffon}
\footnote{ E-mail: \url{r.m.m.griffon@math.leidenuniv.nl}
-- Date: \today{} 
}  
$\ $ (Universiteit Leiden)
}
\date{}

\renewcommand{\O}{\mathcal{O}}

\begin{document}

\maketitle

\paragraph{Abstract --}We prove an analogue of the Brauer-Siegel theorem for Fermat surfaces  over a finite field $\F_q$. Namely, letting $\ferm_d$ be the Fermat surface of degree $d$ over $\F_q$ and $p_g(\ferm_d)$ be its geometric genus, we show that, 
for $d\to\infty$ ranging over the set of integers coprime with $q$, one has
\[ \log\left( |\Br(\ferm_d)|\cdot \Reg(\ferm_d)\right) \sim \log q^{p_g(\ferm_d)} \sim \frac{\log q}{6} \cdot d^3.\]
Here, $\Br(\ferm_d)$ denotes the Brauer group of $\ferm_d$ and $\Reg(\ferm_d)$ the absolute value of a Gram determinant of the Néron-Severi group $\NS(\ferm_d)$ with respect to the intersection form.

\section*{Introduction}
\pdfbookmark[0]{Introduction}{Introduction} 
\addcontentsline{toc}{section}{Introduction}

\setcounter{section}{1}

\paragraph{} We prove an analogue for certain surfaces over a finite field of the classical Brauer-Siegel theorem, 
which asserts that, in families of number fields $k$ of bounded degree, the product of the class number $h_k$ and of the regulator of units $R_k$ is of order of magnitude $\sqrt{\Delta_k}$, where $\Delta_k$ denotes the absolute value of the discriminant of $k$. More precisely,
\begin{theo}[Brauer-Siegel] \label{theo.BS}
When $k$ runs through a sequence of number fields, whose degrees over $\Q$ are bounded and whose discriminants $\Delta_k$ grow to infinity, one has: 
\[\log(h_k\cdot R_k)\sim \log\sqrt{\Delta_k} \qquad\text{(as $\Delta_k\to\infty$)}. \]
\end{theo}

The proof of Theorem \ref{theo.BS} (see \cite[Chap. XVI]{LANT}) is analytic, in that it uses properties of the Dedekind zeta functions $\zeta_k(s)$ of the number fields $k$ in the sequence. There are two main ingredients to it: the first is the analytic class number formula, which relates $h_k\cdot R_k$ to the residue $\zeta_k^\ast$ of $\zeta_k(s)$ at its pole at $s=1$, and the second is the asymptotic estimate $\log\zeta_k^\ast = o\big(\log\sqrt{\Delta_k}\big)$ 
obtained by studying the behaviour of $\zeta_k(s)$ around $s=1$.

In the analogy between number fields and function fields of curves over finite fields, Theorem \ref{theo.BS} has the following translation (see \cite{Inaba}, or \cite{LutharGogia} for a similar statement):

\begin{theo}[Inaba]\label{theo.BS.curves} Given a finite field $\F_q$, when $C$ runs through a sequence of smooth projective and geometrically connected curves over $\F_q$, whose gonalities are bounded and whose genera $g_C$ grow to infinity, one has 
\[ \log \big|\mathrm{Jac}_C(\F_q)\big|
\sim \log q^{g_C}, \qquad\text{(as $g_C\to\infty$)},
\]
where $\mathrm{Jac}_C$ denotes the Jacobian variety of $C$.
\end{theo}

The analogy with Theorem \ref{theo.BS} becomes evident upon noting that $\mathrm{Jac}_C(\F_q)$ is isomorphic to the divisor class group of $\F_q(C)$, and that no regulator of units appears in this setting. 
The proof of Theorem \ref{theo.BS.curves} is analytic too, this time using the Hasse-Weil zeta functions $\zeta(C/\F_q, s)$ of  the curves $C$ in the sequence. 
Both Theorem \ref{theo.BS} and Theorem \ref{theo.BS.curves} have been recently generalized: for example, see \cite{TsfasmanVladut} for a study of the consequences of weakening the hypothesis of bounded degree in Theorem \ref{theo.BS} (respectively, of bounded gonality in Theorem \ref{theo.BS.curves}).

\paragraph{}  
In this article, we prove an analogue of Theorems \ref{theo.BS} and \ref{theo.BS.curves} for a sequence of surfaces over a finite field.
More precisely, given a finite field $\F_q$, consider for all integers $d\geq 2$, 
the Fermat surface $\ferm_d$ of degree $d$, \ie{} the hypersurface of $\P^3$ over $\F_q$ given by:
\[\ferm_d: \qquad X_0^d+X_1^d +X_2^d+X_3^d=0.\]
To ensure that $\ferm_d$ is smooth and geometrically irreducible, we always assume that $d$ is prime to $q$.  
Let $\NS(\ferm_d)$ be the Néron-Severi group of $\ferm_d$ over $\F_q$, it is known to be a finitely generated torsion-free abelian group (see \cite[Chap. V]{Milne_EtCoh} and \cite{Shioda_jacobi}).
It is endowed with the intersection form (a nondegenerate bilinear pairing), and 
one can then 
define the \emph{regulator of $\ferm_d$} to be the Gram determinant 
\[ \Reg(\ferm_d):= \left| \det\big[ (C_i \cdot C_j) \big]_{1\leq i,j\leq \rho}  \right|\in\Z, \]
where $C_1, \dots, C_\rho$ are divisors on $\ferm_d$ whose classes form a $\Z$-basis of $\NS(\ferm_d)$. This construction is reminiscent of the definition of the regulator of units of a number field  as a determinant.

Besides, recall that the Brauer group $\Br(\ferm_d)$ is defined as the group of similarity classes of Azumaya algebras over $\ferm_d$ (see \cite{Groth_BrauerI}, \cite{Groth_BrauerII}, \cite{Milne_EtCoh}).
For our purpose, it is sufficient to know that $\Br(\ferm_d)$ classifies algebraic objects on $\ferm_d$ that are everywhere locally trivial but not globally trivial. 
It is thus a distant relative of the class group of a number field $k$, which classifies ideals that are everywhere locally principal, but not necessarily globally so. 
The Brauer groups of Fermat surfaces has been shown to be finite (see \cite{Shioda_Katsura}, \cite{Tate_BSD}, \cite{Milne_conjAT}).

\paragraph{} Our main result in this article is the following asymptotic estimate on the product $|\Br(\ferm_d)|\cdot \Reg(\ferm_d)$ in terms of the geometric genus $p_g(\ferm_d)$ of $\ferm_d$:
\begin{theo}\label{theo.BS.FERMAT}
 Let $\F_q$ be a finite field and, for any integer $d\geq 2$ that is coprime to $q$, let $\ferm_d$ be the $d$-th Fermat surface over $\F_q$. With notations as above, one has
\begin{equation}\label{eq.intro.BS.FERMAT}
\log\big( \left|\Br(\ferm_d)\right|\cdot\Reg(\ferm_d)\big) 
 \sim \log q^{p_g(\ferm_d)} 
 \qquad (d\to\infty),
 \end{equation}
where $p_g(\ferm_d)$ denotes geometric genus of $\ferm_d$.
\end{theo}
It then follows from an easy estimate of $p_g(\ferm_d)$ that, when $d\to\infty$, one has
\[\log\big( \left|\Br(\ferm_d)\right|\cdot\Reg(\ferm_d)\big) 
 \sim \frac{\log q}{6}\cdot d^3. \]
This  shows that the product 
$|\Br(\ferm_d)|\cdot\Reg(\ferm_d)$ grows 
 exponentially fast with $d$.

 \paragraph{} The proof of Theorem \ref{theo.BS.FERMAT} 
proceeds in two steps. For any integer $d\geq 2$ coprime with $q$, denote by  $\zeta(\ferm_d/\F_q, s)$ the Hasse-Weil zeta function of $\ferm_d$. 
Putting $T=q^{-s}$, Weil's work \cite{weil49} shows that:
\begin{equation}\label{eq.intro.weilzeta}
\zeta(\ferm_d/\F_q, s) = \frac{1}{(1-T)\cdot P_2(\ferm_d/\F_q, T)\cdot(1-q^2T)},
\end{equation}
where $P_2(\ferm_d/\F_q, T)$ is an explicit polynomial with integral coefficients (see Section \ref{sec.zeta.ferm} for details). The Artin-Tate conjecture for surfaces (see \cite{Tate_BSD}, \cite{Tate_conj}, \cite{Milne_conjAT}) provides a (conjectural) analogue of the class number formula (see Section \ref{sec.AT}). For the  surfaces $\ferm_d$, this conjecture has been fully proved by Katsura and Shioda (see \cite{Shioda_Katsura}, \cite{Shioda_jacobi}). Recall that the special value of $P_2(\ferm_d/\F_q, T)$ at $T=q^{-1}$ is defined as follows: 
we let $\rho:=\ord_{T=q^{-1}}P_2(\ferm_d/\F_q, T)$ and put 
\[P_2^\ast(\ferm_d/\F_q):= \left. \frac{P_2(\ferm_d/\F_q, T)}{(1-qT)^{\rho }}\right|_{T=q^{-1}} \in\Z[q^{-1}]\subset\Q.
\]
The Artin-Tate conjecture implies that $\rho = \mathrm{rk} \NS(\ferm_d)$ and 
that $P_2^\ast(\ferm_d/\F_q)$ has the following interpretation (see Section \ref{sec.AT.ferm}): 
\begin{equation}\label{eq.intro.AT.ferm}
P_2^\ast(\ferm_d/\F_q) = \frac{|\Br(\ferm_d)|\cdot\Reg(\ferm_d)}{q^{p_g(\ferm_d)}}.
\end{equation}

The second step of the proof of Theorem \ref{theo.BS.FERMAT} is thus to find 
bounds on $P_2^\ast(\ferm_d/\F_q)$. More precisely, \eqref{eq.intro.AT.ferm} tells us that Theorem \ref{theo.BS.FERMAT} will follow from an estimate of the shape
${\log P_2^\ast(\ferm_d/\F_q)= o \big(\log q^{p_g(\ferm_d)} \big)}$ (when $d\to\infty$). We actually prove a more precise statement:  

\begin{theo}\label{theo.BOUND.SPVAL} Let $\F_q$ be a finite field of characteristic $p$ and $\epsilon\in(0, 1/4)$. For any integer $d\geq 2$ prime to $q$, as $d\to\infty$, one has: 
\begin{equation}\label{eq.intro.BOUND}
- c_1\cdot \left(\frac{\log\log d}{\log d}\right)^{1/4-\epsilon}
\leq \frac{\log P_2^\ast(\ferm_d/\F_q)}{\log  q^{p_g(\ferm_d)} } \leq c_2\cdot \frac{\log\log d}{\log d},
\end{equation}
where $c_1, c_2$ are  positive constants that depend at most on $q$, $p$ and $\epsilon$.
\end{theo}

Theorem \ref{theo.BOUND.SPVAL} is a by-product of our main technical result (see Corollary \ref{theo.BNDS.GEN}), whose  proof occupies most of sections \ref{sec.upperbound} and \ref{sec.lowerbound}. The proof hinges on the explicit expression of $P_2(\ferm_d/\F_q, T)$ obtained by Weil. 
The upper bound in \eqref{eq.intro.BOUND} is 
relatively straightforward (see section  \ref{sec.upperbound}) but the lower bound is more demanding: 
let us give a rough sketch of our strategy (see section \ref{sec.lowerbound}). By construction, the special value has the shape: 
\begin{equation}\label{eq.intro.spval}
P_2^\ast(\ferm_d/\F_q) = \frac{\text{(integer prime to $q$)}}{q^{w_q(d)}} \quad \text{for an exponent } w_q(d)\in\Z_{\geq 0}.
\end{equation}
Upper bounds on $w_q(d)$ in terms of $p_g(\ferm_d)$ imply lower bounds on $\log P_2^\ast(\ferm_d/\F_q)$: to prove the one in \eqref{eq.intro.BOUND}, we are to show that $w_q(d) = o \big(p_q(\ferm_d)\big)$ (as $d\to\infty$). An argument  ``à la Liouville'' gives the trivial upper bound: 
$w_q(d) = O \big(p_q(\ferm_d)\big)$ (see Proposition \ref{prop.triv.lowerbnd}), and we improve on this as follows. The special value $P_2^\ast(\ferm_d/\F_q)$ is given as a product of algebraic numbers related to Jacobi sums: 
we keep track of the contribution of each factor of this product to the denominator in \eqref{eq.intro.spval} by making use of Stickelberger's theorem (see \cite{IR}, \cite{LANT}). In Theorem \ref{theo.proto.lowerbnd}, we thus obtain an explicit expression of $w_q(d)$ in terms of combinatorial data related to the action of $q$ on $\Z/d\Z$ by multiplication. 
To conclude that $w_q(d) = o\big(p_q(\ferm_d)\big)$, we use 
an equidistribution statement which is proved in section \ref{sec.equidis} (Theorem \ref{theo.equidis}). This last step is carried out in Theorem \ref{theo.lowerbnd} and Corollary \ref{theo.BNDS.GEN}. 

In passing, we prove bounds on the rank of $\NS(\ferm_d)$ (see Corollary \ref{coro.rank.bound}):
\begin{theo} \label{theo.intro.rank.bound}
For a given finite field $\F_q$, let $\ferm_d$ be the $d$-th Fermat surface over $\F_q$. We denote by $\rho(\ferm_d/\F_q)$ the rank of the Néron-Severi group of $\ferm_d/\F_q$. One has 
\begin{equation}\label{eq.intro.rank.bound}
\rho(\ferm_d/\F_q)\ll_q \frac{d}{\log d} \qquad (d\to\infty).
\end{equation}
The implied constant is effective and depends only on $q$. 
\end{theo}

This bound on the rank of Néron-Severi groups appears to be new: it improves greatly on the ``geometric'' rank bound of Igusa (see \cite{Igusa}), which would yield
that $\rho(\ferm_d/\F_q)  \leq \dim \H^2_{\acute{e}t}\left(\bar{\ferm_d}, \Q_\ell \right) \leq (d-1)^3.$
Moreover (Proposition \ref{prop.rankbound.optimal}), the bound \eqref{eq.intro.rank.bound} is ``asymptotically optimal'' in the sense that there exists an infinite sequence of integers $d'$ prime to $q$ such that
\[ \rho(\ferm_{d'}/\F_q)\gg_q \frac{{d'}}{\log {d'}} \to +\infty\qquad ({d'}\to\infty). \]

\paragraph{General notations} 
For any finite set $X$, we  denote the cardinality of $X$ by $|X|$.
If $f(x), g(x)$ are functions of a variable $x$ going to $\infty$, $f(x)\sim g(x)$ means that $\lim_{x\to\infty} f(x)/g(x) =1$; and $f(x)\ll_{a} g(x)$ means that $f(x) \leq C_a g(x)$ for some real constant $C_a>0$ depending at most on a parameter $a$.

If $q\geq 1$ and $d\geq 2$ are coprime integers, we let $\sbgp{q}{d} \subset (\Z/d\Z)^\times$ denote the subgroup generated by $q\bmod{d}$. 
The order of $\sbgp{q}{d}$, \ie{} the multiplicative order of $q\bmod{d}$, is denoted by $o_q(d)$. We write $\ord_p(.)$ to designate the $p$-adic valuation on $\Q$.

\paragraph{Acknowledgements} This work is based  on part of the author's PhD thesis \cite{GriffonPHD}. The author wishes to thank his advisor Marc Hindry for his support and many illuminating conversations.
He also thanks Peter Stevenhagen, Michael Tsfasman and Douglas Ulmer for carefully reading drafts of this work, 
and  for many comments and corrections. 
Thanks are also due to Igor Shparlinski for pointing out a mistake in an earlier version. 
This article was written while being a postdoc at Universiteit Leiden, whose financial support 
is gratefully acknowledged. 

\numberwithin{equation}{section}

\section{Special values of zeta functions of Fermat surfaces}

In this section, we quickly review useful facts about zeta functions of surfaces over finite fields and conjectures about them. We also recall the definitions of Fermat surfaces and known results about their zeta functions.

Let $\F_q$ be a finite field of characteristic $p$, and let $\Scal$ be a smooth projective and geometrically irreducible surface over $\F_q$. We write $\bar{\Scal} = \Scal \times_{\F_q}\bar{\F_q}$, where $\bar{\F_q}$ is the algebraic closure of $\F_q$. For a  prime number $\ell\neq p$, denote by $\H^i(\Scal)$ the $i$-th $\ell$-adic étale cohomology space $\H^i_{\acute{e}t}(\bar{\Scal}, \Q_\ell)$.

The \emph{Néron-Severi group} $\NS(\Scal)$ is defined to be the image of $\Pic(\Scal)$ in $\NS(\bar{\Scal})$. Since the base field $\F_q$ is finite, another possible definition is $\NS(\Scal):=\NS(\bar{\Scal})^{\Gal(\bar{\F_q}/\F_q)}$, the ${\Gal(\bar{\F_q}/\F_q)}$-invariant subgroup of $\NS(\bar{\Scal})$ (see \cite[Prop. 6.2]{PTvL}).
The so-called ``Theorem of the base'' asserts that $\NS(\bar{\Scal})$, and therefore $\NS(\Scal)$, is a finitely generated abelian group (see 
\cite[Chap. V]{Milne_EtCoh}). The rank $\rho(\Scal/\F_q)$ of $\NS(\Scal)$ is called the Picard number of $\Scal$. 
Moreover, the Néron-Severi group is endowed with the intersection form $(-\cdot-)$, which is a nondegenerate bilinear pairing: if $D_1, \dots, D_\rho$ is a set of divisors on $\Scal$ whose classes generate $\NS(\Scal)$ modulo torsion, we define the \emph{regulator of $\Scal$} to be
\begin{equation}\label{eq.defi.Reg}
 \Reg(\Scal/\F_q):= \left| \det\big(D_i\cdot D_j \big)_{1\leq i,j\leq \rho}\right|\in\Z.
\end{equation}

The \emph{Brauer group of $\Scal$} can be defined in (at least) two ways: one can define $\Br(\Scal/\F_q)$ as the group of similarity classes of Azumaya algebras over $\Scal$ (see \cite{Groth_BrauerI}) which, when $\Scal$ is smooth and projective, is isomorphic to the ``cohomological'' Brauer group 
\begin{equation}\label{eq.defi.Br}
\Br(\Scal/\F_q):= \H^2_{\acute{e}t}(\Scal, \mathbb{G}_m) = \H^2_{\acute{e}t}(\Scal, \O_{\Scal}^\times),
\end{equation}
see \cite{Groth_BrauerII}, \cite[Chap. V]{Milne_EtCoh}. It is conjectured, but not known in general, that $\Br(\Scal/\F_q)$ is finite (see \cite{Tate_BSD}).

\subsection{Zeta functions of surfaces and their special values}\label{sec.AT}

\emph{A priori}, the zeta function of $\Scal$ is defined as the formal power series 
\[ Z(\Scal/\F_q, T):= \exp\left(\sum_{n=1}^{+\infty}\frac{|\Scal(\F_{q^n})|}{n}\cdot T^n\right)\]
where $|\Scal(\F_{q^n})|$ denotes the number of $\F_{q^n}$-rational points on $\Scal$. Since Weil's conjectures were proved by Deligne \cite{Deligne_Weil1}, it is known that 
$Z(\Scal/\F_q, T)$ is actually a rational function of $T$: more precisely, 
\[ Z(\Scal/\F_q, T) = \frac{P_1(T)\cdot P_3(T)}{P_0(T)\cdot P_2(T)\cdot P_4(T)}, \]
with 
$P_i(T) 
= \det\left(1-\fr^\ast \cdot T \  \big|\  \H^i(\Scal) \right),$
where $\fr^\ast$ is the endomorphism of $\H^i(\Scal)$ induced 
by the action of the geometric Frobenius on $\bar{\Scal}$. 
It is also known that the polynomials $P_i(T)$ have integral coefficients, are independent of the choice of $\ell\neq p$, and that their reciprocal roots are algebraic integers of absolute value $q^{i/2}$ in any complex embedding (the so-called Riemann hypothesis for $Z(\Scal/\F_q, T)$). For more details about these facts, the reader can consult \cite{Milne_EtCoh}. 

If $\Scal$ is geometrically irreducible, one has $P_0(T) = 1-T$ and $P_4(T) = 1-q^2T$. 
Furthermore, by Poincaré duality, $P_3(T)=P_1(qT)$ and, as soon as $\Scal$ has a trivial Picard variety, $P_1(T) = P_3(T)=1$. 
Finally, any  nonsingular surface $\Scal$ of degree $d$ in $\P^3$ has 
\begin{equation}\label{eq.b2}
\deg P_2(T)= \dim \H^2(\Scal) 
= (d-1)(d^2-3d+3)+1.
\end{equation}

\paragraph{} 
Let us study the analytic behaviour of $P_2(\Scal/\F_q, T)$ at $T=q^{-1}$ (that is, the behaviour of $s \mapsto Z(\Scal/\F_q, q^{-s})$ at $s=1$).  
First, define the \emph{analytic rank} $\rho$ of $\Scal$
 to be the order of vanishing of $P_2(T)$ at $T=q^{-1}$. Secondly, define the \emph{special value at $T=q^{-1}$} to be 
\[ P_2^\ast(\Scal/\F_q, q^{-1}):= \left.\frac{P_2(\Scal/\F_q, T)}{(1-qT)^{\rho}}\right|_{T=q^{-1}}.
\]
By definition, $P_2^\ast(\Scal/\F_q, q^{-1})$ is a nonzero rational number, and the Riemann hypothesis for $P_2(\Scal/\F_q, T)$ implies that $P_2^\ast(\Scal/\F_q, q^{-1})$ is positive.

Inspired by the conjectures of Birch and Swinnerton-Dyer for elliptic curves, Tate and Artin-Tate conjectured that these quantities have an ``arithmetic'' interpretation: 

\begin{conj}[Artin-Tate]\label{conj.AT}
 Let $\Scal$ be a projective smooth and geometrically irreducible surface over a finite field $\F_q$, $p_g(\Scal)$ be the geometric genus of $\Scal$ and let $\rho(\Scal/\F_q)$ denote the rank of the Néron-Severi group $\NS(\Scal)$ of $\Scal$. Then
\begin{enumerate}[(1)]
\item The Néron-Severi group has rank $\displaystyle \rho(\Scal/\F_q) =  \ord_{T=q^{-1}}P_2(\Scal/\F_q, T),$
\item The Brauer group $\Br(\Scal)$ is finite,
\item The special value $P^\ast_2(\Scal/\F_q, q^{-1})$ admits the expression:
\begin{equation}\label{eq.conj.AT}
P_2(\Scal/\F_q, q^{-1}) = \frac{|\Br(\Scal/\F_q)|\cdot\Reg(\Scal/\F_q)}{q^{p_g(\Scal)} \cdot |\NS(\Scal)\tors|^2} \cdot q^{\delta(\Scal)},
\end{equation}
where 
$\delta(\Scal) = \dim \H^1(\Scal, \O_{\Scal}) - \dim \Pic(\Scal)$ is the ``defect of smoothness'' of the Picard variety of $\Scal$ (it is known that $\delta(\Scal)\leq p_g(\Scal)$, see \cite{Tate_BSD}).
\end{enumerate}
\end{conj}
Equality \eqref{eq.conj.AT} can also be written under the equivalent form of a Taylor expansion of $P_2(\Scal/\F_q, q^{-s})$ around $s=1$:
\[ P_2(\Scal/\F_q, q^{-s}) = \frac{|\Br(\Scal/\F_q)|\cdot\Reg(\Scal/\F_q) \cdot q^{\delta(\Scal)}}{q^{p_g(\Scal)} \cdot |\NS(\Scal)\tors|^2} \cdot (1-q^{1-s})^{\rho} + O\left((1-q^{1-s})^{\rho+1}\right) \quad (\text{as } s\to 1).
\]

We refer to \cite[Conjecture C]{Tate_BSD} for the original statement, and to \cite[§4]{Tate_conj} for a more general one. Tate and Milne have subsequently proved that parts (1) to (3) are equivalent (see \cite{Tate_conj}, \cite{Milne_conjAT}, or \cite{UlmerCRM} for a survey). 
The full Conjecture \ref{conj.AT} is known for certain surfaces over $\F_q$, among which Fermat surfaces (see below). 

\subsection{Fermat surfaces}
\label{sec.fermatsurf}

For any positive integer $d$ coprime to $p$, let $\ferm_d$ be the \emph{$d$-th Fermat surface} over $\F_q$, whose equation in $\P^3$ is
\begin{equation}\label{eq.fermatsurf}
 \ferm_d: \qquad X_0^d+X_1^d+X_2^d + X_3^d=0.
\end{equation}
Thus defined, $\ferm_d$ is a smooth, projective and geometrically irreducible surface (since $d$ is prime to $q$). These surfaces -- and their higher dimensional analogues -- have been studied in great detail, in particular by Shioda and his coauthors (see, in particular,  \cite{Shioda_Katsura}, \cite{Shioda_picardalgo}, \cite{Shioda_jacobi}). Let us recall the facts about $\ferm_d$ that are most relevant to our present goal. Like any nonsingular surface of degree $d$ in $\P^3$, the Fermat surface $\ferm_d$ has geometric genus
\[p_g(\ferm_d) = \frac{(d-1)(d-2)(d-3)}{6} = \binom{d-1}{3}.\]
In particular, $p_g(\ferm_d) \sim d^3/6$ as $d\to\infty$.
Furthermore, for any nonsingular surface $\Scal$ in $\P^3$, one knows that $\NS(\Scal)$ is torsion-free, and that the Picard variety $\Pic(\Scal)$ is trivial (so that, in particular, the defect of smoothness $\delta(\Scal)$ vanishes). See \cite{Shioda_Katsura}, \cite{Shioda_picardalgo}, \cite{Shioda_jacobi} and \cite{SSvL} for detailed geometric information on $\ferm_d$.

Crucial to our further study is the natural action of
$\Gamma_d:= \mu_d(\bar{\F_q})^4/\text{(diagonal)}\subset\mathrm{Aut}_{\bar{\F_q}}(\ferm_d)$ on $\ferm_d$ via 
\[
\forall \xx=[x_0:x_1:x_2:x_3]\in\ferm_d, 
\forall \zzeta=[\zeta_0:\zeta_1:\zeta_2:\zeta_3]\in\Gamma_d, \quad  
\zzeta \cdot \xx = [\zeta_0x_0: \zeta_1x_1: \zeta_2x_2:\zeta_3x_3 ].\]
Note that, when $d$ divides $|\F_q^\times| = q-1$, the action of $\Gamma_d$ on $\ferm_d$ is an action by $\F_q$-automorphisms (since then, all $d$-roots of unity in $\bar{\F_q}$ are $\F_q$-rational). In general though (\ie{} $d$ only assumed to be prime to $q$), the action of the $q$-th power Frobenius $\fr_q$ on $\Gamma_d$ is not trivial. Instead of $\Gamma_d$, we will rather study its character group $\hhat{\Gamma_d}$ 
under the following ``combinatorial'' incarnation:
\begin{equation}\label{eq.gammadhat}
G_d:= \left\{ \ba = (a_0, \dots, a_3)\in(\Z/d\Z)^4 \ \big| \ a_0 + \dots + a_3 \equiv 0 \bmod{d}\right\} \simeq \hhat{\Gamma_d}.
\end{equation}

\subsection{Zeta functions of Fermat surfaces} \label{sec.zeta.ferm}
\label{sec.notations}
  \paragraph{} The zeta function of Fermat surfaces have been explicitly computed by Weil in \cite{weil49} as evidence for his   conjectures. 
We recall his result in this subsection 
but before we do so, we need to introduce a few more notations which will be in force for the rest of the paper.

\subsubsection[Action]{Action of $q$ on $G_d$, and the Teichmüller character} 
\label{subsec.notations}
 Let $p$ be a prime number; 
 we fix, once and for all, an algebraic closure $\Qbar$ of $\Q$ (of which all number fields are seen as subfields) and a prime ideal $\gP$ above $p$ in the ring of integers $\bar{\Z}$ of $\Qbar$. The quotient field $\bar{\Z}/\gP$ is then an algebraic closure of $\F_p$: all the finite fields $\F_q$ of characteristic $p$ 
 involved in our computations will be seen as subfields of this algebraic closure. Let $\mu_{p'}\subset\Qbar$ be the group of roots of unity whose order is prime to $p$. The reduction map $\bar{\Z}\to\bar{\Z}/\gP =\bar{\F_p}$ induces an isomorphism $\mu_{p'} \to \bar{\F_p}^\times$. We denote by $\tt: \bar{\F_p}^\times\to\mu_{p'}\hookrightarrow \Qbar^\times$ the inverse of this isomorphism: we call $\tt$  the \emph{Teichmüller character}, and we also denote by $\tt$ the restrictions of $\tt$ to the various finite fields $\F_q\subset\bar{\F_p}$.

 \paragraph{} There is a natural action of $(\Z/d\Z)^\times$ on $G_d$ by component-wise multiplication: for all $t\in(\Z/d\Z)^\times$ and all $\ba=(a_0, \dots, a_3)\in G_d$, we let $t\cdot \ba:=(ta_0, \dots, ta_3)$. 
 Since $q$ is prime to $d$, the subgroup $\sbgp{q}{d}\subset(\Z/d\Z)^\times$ generated by $q$ also acts on $G_d$ by multiplication. 
 
For any subset $\Lambda$ of $G_d$ which is stable under this action of $q$, we will denote by $\O_q(\Lambda)$ the set of orbits of $\Lambda$ under multiplication by $q$.
Given an orbit $\bA\in\O_q(\Lambda)$, we will often have to make a choice of a representative $\ba\in\Lambda$ of this orbit. To avoid repeating the sentence ``[...] where $\ba$ is a representative of the orbit $\bA$'', we will stick to the following convention: orbits in $\O_q(\Lambda)$ will always be denoted by an \emph{uppercase bold} letter ($\bA, \bB$,~...) and we denote by the corresponding \emph{lowercase bold} letter ($\ba$, $\bb$, ...) any choice of a representative in $\Lambda$ of that orbit.
For any orbit $\bA\in\O_q(G_d)$, we let $|\bA|$ be the length of $\bA$; in other words, for any representative $\ba=(a_0, \dots, a_3)\in\bA$, one has $|\bA| = \big| \{\ba, q\cdot \ba, \dots, q^{n} \cdot \ba, \dots\}\big|$, or equivalently 
\[ |\bA|  = \min\left\{n\in\Z_{\geq 1} \ \big|\ \forall i\in\{0,1,2,3\},\  q^n\cdot a_i \equiv a_i \bmod{d}\right\}.\]
For any $\ba\in G_d$, let $d_\ba:=d\big/\gcd(d, a_0, \dots, a_3)$: since $q$ is prime to $d$, it is easy to see that $\ba\mapsto d_\ba$ is constant along an orbit under multiplication by $q$, and 
one checks that $|\bA| 
= o_q(d_{\ba})$ where, for any integer $n\geq 2$ prime to $q$, $o_q(n)$ denotes the (multiplicative) order of $q$ in $(\Z/n\Z)^\times$. For any power $q^v$ of $q$, by definition of the order, we see that $q^v\ba \equiv \ba \bmod{d}$ if and only if $o_q(\ba)$ divides $v$, \ie{} $\F_{q^v}$ is an extension of $\F_{q^{|\bA|}}$.

It is also easily checked that the pairing 
$(\ba, \zzeta)\in G_d \times\Gamma_d\mapsto 
\tt_\ba(\zzeta):= \tt(\zeta_0)^{a_0} \cdot\dotso\cdot\tt(\zeta_3)^{a_3}\in\Qbar^\times $
 induces the isomorphism  $\hhat{\Gamma_d} \simeq G_d$ alluded to in \eqref{eq.gammadhat}. Moreover, this isomorphism takes the $q$-th power Frobenius action on $\Gamma_d$ to the action of $q$ by multiplication on $G_d$ in the following sense:
\[\forall \ba\in G_d, \forall \zzeta =[\zeta_0:  \dots:\zeta_3]\in\Gamma_d, \quad  \tt_{\ba}(\fr_q (\zzeta))= \tt_{\ba}(\zeta_0^q, \dots, \zeta_3^q) = \tt_{q\cdot \ba}(\zeta_0, \dots, \zeta_3) = \tt_{q\cdot \ba}(\zzeta).\]

\subsubsection{Jacobi sums}
To state Weil's result in a convenient form, we need to introduce a specific `instantiation'' of Jacobi sums. We make the convention that characters $\chi: \F_q^\times \to\Qbar^\times$ are extended by $\chi(0)=0$ unless $\chi$ is the trivial character, in which case we put $\chi(0)=1$. Classical facts about characters of finite fields and Jacobi sums can be found in \cite{IR} and \cite{LidlN}.
\begin{defi}\label{defi.jacobi}
For all $\ba=(a_0, \dots, a_3)\in G_d\smallsetminus\{(0,0,0,0)\}$ and  all integers $s\in\Z_{\geq1}$, define the following characters on $\F_Q^\times$ where $Q=q^{s\cdot |\bA|}$:
\[\forall i\in\{0, 1,2,3\}, \ \cchi_{i}: \F_Q^\times \to \Qbar^\times, \ x \mapsto \left(\tt\circ\norm_{\F_Q/\F_{q^{|\bA|}}}(x)\right)^{(q^{|\bA|}-1)\cdot a_i/d}.
\]
One then defines a Jacobi sum (relative to $\F_Q$):
\[ \Ja_{Q}(\ba) = \frac{1}{Q-1}\sum_{\substack{x_0, \dots, x_3 \in\F_{Q}^\times\\ x_0 +\dots + x_3=0}}\cchi_0(x_0)\cchi_1(x_1)\cchi_2(x_2)\cchi_3(x_3).\]
In the case where $s=1$, 
we denote $\Ja_{q^{|\bA|}}(\ba)$ by $\Ja(\ba)$ for short. By convention, let $\Ja(0,0,0,0)=q$.
\end{defi}

This normalization of Jacobi sums is the same as that of Weil \cite{weil49} and of Shioda \cite{Shioda_jacobi}. 
Note that the characters $\cchi_i$ have order dividing $d$ (actually, the order of $\cchi_i$ can be seen to be $d/\gcd(d,a_i)$) and $\cchi_i$ is trivial if and only if $a_i=0$.
In particular, one can see that $\Ja_Q(\ba)$ is an algebraic integer in the cyclotomic field $\Q(\zeta_d)$.
The Galois group $\Gal(\Q(\zeta_d)/\Q)$ thus permutes the Jacobi sums: let us identify $\Gal(\Q(\zeta_d)/\Q)$ with $(\Z/d\Z)^\times$ in the usual manner (with $t\in(\Z/d\Z)^\times$ corresponding to $\sigma_t\in\Gal(\Q(\zeta_d)/\Q)$), then 
\begin{equation}\label{eq.actiongal.jacobi}
\forall \ba\in G_d, \forall t\in(\Z/d\Z)^\times, \quad \sigma_t\left(\Ja_Q(\ba)\right) = \Ja_Q(t\cdot \ba).
\end{equation}

Furthermore, it is well-known that $|\Ja_Q(\ba)| = Q$ if and only if  all $\cchi_i$ are non trivial and their product $\cchi_0\cdot\dotso\cdot\cchi_3$ is trivial: we are thus led to introduce 
\[G_d^\circ:=\left\{\ba=(a_0, \dots, a_3) \in (\Z/d\Z)^4 \ \big| \  \forall i, \ a_i\neq 0 \text{ and } \sum a_i =0 \right\}  = \left\{\ba \in G_d\ \big| \  \forall i, \ a_i\neq 0 \right\}.\]
We note that the subset $G_d^\circ\subset G_d$ is stable under the action of $(\Z/d\Z)^\times$ and has cardinality $|G_d^\circ| = (d-1)(d^2-3d+3) = (d-1)^3-2d^2$. 
Also, we remark that $\Ja(q\cdot \ba) = \Ja(\ba)$, and more generally that $\Ja(p\cdot \ba) = \Ja(\ba)$. 
Finally, we recall the relation of Davenport-Hasse for Jacobi sums in the following form (see \cite{weil49}, \cite{IR}): 
\begin{equation}\label{eq.hassedav.jacobi}
\text{if $\ba\in G_d$  and $s\geq 1$ is an integer, then $\displaystyle \Ja_{q^{s|\bA|}}(\ba) = \left(\Ja_{q^{|\bA|}}(\ba)\right)^{s} = \Ja(\ba)^s$.}
\end{equation}

\subsubsection{Zeta functions of Fermat surfaces}

We have now enough notations to state the result of Weil alluded to earlier:

\begin{theo}[Weil]\label{theo.weil}
Let $\F_q$ be a finite field and $d\geq 2$ be an integer, coprime with $q$. Set $G_d^\circ \subset G_d$ as above. 
The Fermat surface $\ferm_d/\F_q$ defined by equation \eqref{eq.fermatsurf} has zeta function given by 
\begin{equation}\label{eq.zeta.ferm}
Z(\ferm_d/\F_q, T)  = \frac{1}{(1-T)\cdot P_2(\ferm_d/\F_q, T) \cdot (1-q^2T)}, \end{equation}
where, denoting by $\Ja({\ba})= \Ja_{q^{|\bA|}}(a_0, \dots, a_3)$ the Jacobi sum defined above, 
\begin{equation}\label{eq.P2.ferm}
P_2(\ferm_d/\F_q, T)= (1-qT)\cdot 
\prod_{\bA\in\O_q(G^\circ_d)} \left(1-\Ja(\ba)\cdot T^{|\bA|} \right). 
\end{equation}
\end{theo}

There are at least two ways to obtain this expression: one is by a ``point-counting'' argument (see \cite{weil49} or \cite{IR}), another is via a more cohomological method (see \cite{Shioda_Katsura} or \cite[Coroll. 2.4]{Katz_crys}). 
Note that the ``usual'' setting for the proof is the hypothesis that $d$ divides $q-1$, which is insufficient for our use
since we need $d$ to be arbitrarily large with respect to a fixed $q$.
This explains the appearance of the action of $q$ on the indexing set $G_d^\circ$ for the Jacobi sums: multiplication by $q$ on $G_d$ is the ``combinatorial'' version of the action of $\Gal(\bar{\F_q}/\F_q)$ on~$\Gamma_d$ (see section \ref{sec.fermatsurf}).

\begin{rema}If one considers the Fermat surface ``with coefficients'', that is to say the surface defined by \[\ferm'_d: c_0X_0^d + c_1X_1^d+ c_2X_2^d+ c_3X_3^d=0
\quad \text{ with } c_i\in\F_q^\times,\]
one can also give an explicit expression of the zeta function. The only change in \eqref{eq.zeta.ferm} is that  
\[P_2(\ferm'_d/\F_q, T) = (1-qT) \cdot
\prod_{\bA\in\O_q(G^\circ_d)} \left(1-\xi_{\bc}(\ba) \cdot\Ja(\ba)\cdot T^{|\bA|} \right),
\]
where $\xi_{\bc} (\ba) =\prod_{i=0}^3 \cchi_i(c_i^{-1})$ 
is a $d$-th root of unity in $\Qbar$. The reader can check that our arguments to bound $P_2^\ast(\ferm_d/\F_q, q^{-1})$ would work just as well for $\ferm'_d$ (uniformly in the choice of coefficients $c_i$).
\end{rema}

\subsection{Special values of zeta functions of Fermat surfaces}\label{sec.AT.ferm}
 
We now introduce $G_d^\ast:= \left\{\ba \in G_d^\circ \ \big| \ \Ja(\ba)\neq q^{|\bA|}\right\}$, the subset of $G_d^\circ$ parametrizing the factors ${1-\Ja(\ba) T^{|\bA|}}$ of $P_2(\ferm_d/\F_q, T)$ that don't vanish at $T=q^{-1}$: this set $G_d^\ast$ is obviously stable under the action of $q$; furthermore, a computation (as in Lemma \ref{lemm.expr.spval}) shows that the special value of $P_2(\ferm_d/\F_q, T)$ at $T=q^{-1}$ is:
\begin{equation}\label{eq.expr.spval.fermat}
P_2^\ast(\ferm_d/\F_q, q^{-1}) = \prod_{\ba\in \O_q(G_d^\circ\smallsetminus G_d^\ast)} |\bA| \cdot  \prod_{\bA\in\O_q(G_d^\ast)}\left(1-\frac{\Ja(\ba)}{q^{|\bA|}} \right).
\end{equation}

On the other hand, $P_2^\ast(\ferm_d/\F_q, q^{-1})$ also has an expression in terms of algebro-geometric invariants of $\ferm_d$, indeed:

\begin{theo}[Shioda]\label{theo.shioda.AT}
 Over $\F_q$, the Fermat surface $\ferm_d$ satisfies the Artin-Tate conjecture (Conjecture \ref{conj.AT}). In particular, its Brauer group $\Br(\ferm_d)$ is finite.
\end{theo}

Shioda actually proved that  $\ferm_d/\F_q$ satisfies the Tate conjecture (part (1) of Conjecture \ref{conj.AT}). 
His proof relies on the following two facts: first, if one knows that (1) holds for a surface $\Scal$ and if there is a dominant rational map $\Scal \dashrightarrow \Scal'$, then (1) also holds for $\Scal'$.
Second, (1) is true for surfaces $\Scal$ that are products of curves $C_1\times C_2$ by a famous result of Tate \cite{Tate_conj}.
In particular, (1) is known for all surfaces that are dominated by products of curves (note that not all surfaces are of this form). 
Now Katsura and Shioda explicitly constructed a dominant rational map $C_1\times C_2\dashrightarrow \ferm_d$ from a product of Fermat curves to $\ferm_d$. 
For more details, see \cite{Shioda_Katsura}.

The veracity of Conjecture \ref{conj.AT} for $\ferm_d$ (in particular, of part (3) of Conjecture \ref{conj.AT}) yields the following expression for $P_2^\ast(\ferm_d/\F_q, q^{-1})$ 
(see Proposition 5.2 of \cite{Shioda_jacobi}):
\begin{prop}[Shioda]\label{prop.fermat.spval}
Let $\F_q$ be a finite field. 
  For any integer $d\geq 2$ that is prime to $q$, one has:
\begin{equation}
P_2^\ast(\ferm_d/\F_q, q^{-1}) = \frac{|\Br(\ferm_d)|\cdot \Reg(\ferm_d)}{q^{p_g(\ferm_d)}}. \end{equation}
\end{prop}

\section{A more general setting}

With Proposition \ref{prop.fermat.spval} and the explicit expression \eqref{eq.expr.spval.fermat} for the special value $P_2^\ast(\ferm_d/\F_q, q^{-1})$ of the zeta function of the Fermat surfaces, the proof of our main theorem (Theorem \ref{theo.BS.FERMAT}) reduces to proving the asymptotic bound on the special value $P_2^\ast(\ferm_d/\F_q, q^{-1})$ mentioned in the introduction (Theorem \ref{theo.BOUND.SPVAL}).
Since we have other applications   in mind for these bounds on special values, and since it won't lead to too many technical complications, we will consider a slightly more general setting which we now describe.

\subsection[Polynomials and special values]{The polynomials $P(\Lambda, T)$}
\label{sec.setting}

Let $\F_q$ be  a finite field of characteristic $p$. As above, for any integer $d\geq 2$ that is prime to $p$, we let 
\[G_d = \big\{(a_0, \dots, a_3)\in(\Z/d\Z)^4 \ | \ \textstyle{\sum_{i=0}^3 a_i =0}\big\} \subset (\Z/d\Z)^4.\] 
Recall that $(\Z/d\Z)^\times$ (and thus its subgroup $\sbgp{q}{d}$) acts by multiplication on $G_d$. If $\Lambda \subset G_d$ is stable by multiplication by $q$, we denote by $\O_q(\Lambda)$ the set of orbits of $\Lambda$ under this action. We will be interested in the following class of polynomials:

\begin{defi}\label{defi.polyn.lambda}
Let $\Lambda$ be a  nonempty subset of $G_d$. We assume that $\Lambda$ 
is stable under the action of $(\Z/d\Z)^\times$ by multiplication. We then define the following polynomial 
\[ P(\Lambda, T) = \prod_{\bA \in \O_q(\Lambda)} \left( 1- \Ja(\ba) \cdot T^{|\bA|}\right) \in \Z[T],\]
where $\Ja(\ba) = \Ja_{q^{|\bA|}}(a_0, \dots, a_3)$ is the Jacobi sum defined above (Definition \ref{defi.jacobi}).
\end{defi}
Under the assumption that $\Lambda$ is $(\Z/d\Z)^\times$-stable, $P(\Lambda, T)$ indeed has integral coefficients (because the Jacobi sums are algebraic integers and the action of $\Gal(\Q(\zeta_d)/\Q)$ on $\{\Ja(\ba)\}_{\ba\in G_d}$ corresponds to the action of $(\Z/d\Z)^\times$ on $G_d$, see \eqref{eq.actiongal.jacobi}). Note that $P(\Lambda, T)$ implicitly depends on $q$, but we chose not to include it in the notation since $q$ is fixed. Besides, we remark that $\deg P(\Lambda, T) =\left|\left\{\ba\in\Lambda \ \big| \ \Ja_\ba\neq 0\right\}\right| \leq |\Lambda|$.

Now define the special value of $P(\Lambda, T)$ at $T=q^{-1}$ to be
\[ P^\ast(\Lambda):= \left. \frac{P(\Lambda, T)}{(1-qT)^{\rho}}\right|_{T=q^{-1}}, \text{ where } \rho = \ord_{T=q^{-1}}P(\Lambda, T).\]
In other words, if we set $P^\ast_\Lambda(T):= P(\Lambda, T)\cdot (1-qT)^{-\rho} \in\Z[T]$, 
then $P^\ast(\Lambda) = P^\ast_\Lambda(q^{-1})$. By construction, $P^\ast(\Lambda)$ is a nonzero element of $\Z[q^{-1}]\subset\Q$.

\subsection{Examples}

To justify our considering such objects, let us give a few examples of situations in which polynomials $P(\Lambda, T)$ naturally appear. 

\begin{exple}\label{ex.fermat.P2}
 As a first example, let us choose $\Lambda_\ferm = G_d$ for any integer $d\geq 2$ coprime with $q$. 
 For any $\ba=(a_0, a_1, a_2, a_3)\in G_d$, recall that 
\[ \Ja(\ba) = \begin{cases}
 q & \text{ if } \ba =(0,0,0,0), \\
0 & \text{ if some (but not all) of the $a_i$ are $0$}, \\ 
\text{of absolute value $q^{|\bA|}$}& \text{ if } \ba\in G_d^\circ.
 \end{cases}
\]
So that, from Weil's theorem (Theorem \ref{theo.weil}), we obtain
\[P(\Lambda_\ferm, T) = (1-qT)\cdot\prod_{\bA\in\O_q(G_d^\circ)}\left(1-\Ja(\ba)\cdot T^{|\bA|}\right) = P_2(\ferm_d/\F_q, T).\]
Notice that  $\Lambda_\ferm$ is of size $|\Lambda_\ferm| = |G_d| = d^3 \sim 6 \cdot p_g(\ferm_d)$ when $d\to\infty$. This example is actually the one to which we apply our main result (Corollary \ref{theo.BNDS.GEN}) here.
\end{exple}

\begin{exple} For an integer $d\geq 2$ coprime with $q$, consider a subgroup $H$ of $\Gamma_d = \mu_d(\bar{\F_q})^4/\text{(diagonal)}$ 
(see subsection \ref{sec.fermatsurf}). The finite group $\H$ acts on $\ferm_d$, and we let $\Scal:= \ferm_d/H$  be the quotient of the Fermat surface $\ferm_d$ by this action. 
The resulting surface $\Scal$ is defined over $\F_q$ and is normal (but not necessarily smooth). Now, let $\Lambda_H \subset G_d$ be the subgroup of $G_d$ isomorphic to $H^{\perp} = \big\{\chi\in\hhat{\Gamma_d} \ \big|\ \forall h\in H, \ \chi(h)=1 \big\}$ in the isomorphism $\hhat{\Gamma_d} \simeq G_d$. Being a subgroup of $G_d$, $\Lambda_H$ is clearly nonempty and stable under multiplication by $(\Z/d\Z)^\times$.
Denote by $P_2(\Scal/F_q, T)$ the characteristic polynomial of the Frobenius $\fr_q$ acting on $\H^2_{\acute{e}t}(\bar{S}, \Q_\ell)$. A computation very similar to that of \cite[§7]{Ulmer_LargeRk} leads to:
\[P_2(\Scal/\F_q, T) = P(\Lambda_H, T).\]
\end{exple}

\begin{exple}
As a special case, for any integer $d\geq 2$ prime to $q$, consider the subgroup $H \subset \Gamma_d$ 
generated by $[\zeta^2:\zeta:1:1]$ and $[1:\zeta:\zeta^3:1]$, for $\zeta\in\mu_d(\bar{\F_q})$ a primitive $d$-th root of unity. Set $\Scal:= \ferm_d/H$ to be the quotient of the Fermat surface by the action of $H$. 

Ulmer  has proved a number of facts about $\Scal$, among which the identity 
$P_2(\Scal/\F_q, T) = P(\Lambda_H, T),$ (see \cite[§7]{Ulmer_LargeRk})
where
$\Lambda_H \subset G_d$ is associated to $H$ as in the example above. 
The polynomial $P_2(\Scal/\F_q, T)$ is closely related to the L-function of the elliptic curve $E_d/\F_q(t)$ given by $y^2+xy=x^3-t^d$ (see \cite{Ulmer_LargeRk} for details).
For this example, our upper and lower bounds on $P^\ast_2(\Scal/\F_q,q^{-1})$ have been proved by Hindry and Pacheco as the first example of a family of elliptic curves over $\F_q(t)$ unconditionally satisfying an analogue of the Brauer-Siegel theorem (see \cite[§7.4]{HP15}). There are now five more examples of such families, by \cite{GriffonPHD}.
\end{exple}

\subsection{Two preliminary lemmas}

Let $\Lambda\subset G_d$ be a nonempty $(\Z/d\Z)^\times$-stable subset: 
in the lemma below, we give an explicit expression for the special value $P^\ast(\Lambda)$. 
First, let us introduce the following decomposition of $\Lambda$:
\[\Lambda = \Lambda_0 \sqcup \Lambda^\ast \sqcup 
\Lambda',\]
 where  
 $\Lambda_0 = \left\{\ba\in\Lambda \ \big| \ \Ja({\ba}) = q^{|\bA|}\right\} $,  
$\Lambda^\ast = \left\{\ba \in\Lambda \ \big | \ |\Ja({\ba})| = q^{|\bA|} \text{ but }\Ja({\ba})\neq q^{|\bA|}\right\}$, 
and 
$\Lambda' = \Lambda\smallsetminus (\Lambda_0\sqcup \Lambda^\ast 
)$. 
Each of these subsets of $\Lambda$ is stable under the action of $q$ since the value of $\Ja(\ba)$ does not depend on the choice of representative $\ba\in\bA$.
Notice that, if $\bA\in\O_q(\Lambda')$ one has $\Ja({\ba}) = 0$ and that,
 if $(0,0,0,0)\in\Lambda$ then $(0,0,0,0)\in\Lambda_0$.
With this decomposition at hand, we can state:

\begin{lemm}\label{lemm.expr.spval} 
Notations being as above, the multiplicity of $T=q^{-1}$ as a root of $P(\Lambda, T)$ is 
$|\O_q(\Lambda_0)|$, and the special value $P^\ast(\Lambda)$ can be expressed as:
\begin{equation}\label{eq.expr.spval}
P^\ast(\Lambda) = P^\ast_\Lambda(q^{-1})
=
\prod_{\bA\in\O_q(\Lambda_0)} |\bA| \cdot \prod_{\bA\in\O_q(\Lambda^\ast)} \left(1-\frac{\Ja(\ba)}{q^{|\bA|}}\right).
\end{equation}
\end{lemm}

\begin{proof}
For any orbit $\bA\in \O_q(\Lambda)$, we let $g_\bA(T):= 1-\Ja(\ba) T^{|\bA|}$ be the corresponding factor of $P(\Lambda, T)$. The polynomial $g_\bA(T)$ vanishes at $T=q^{-1}$ if and only if $\Ja(\ba)= q^{|\bA|}$, in which case $T=q^{-1}$ is a simple root of $g_{\bA}(T)$. 
This means that 
\[\rho:=\ord_{T=q^{-1}}P(\Lambda, T)=\sum_{\bA} \ord_{T=q^{-1}}g_\bA(T) = \left| \left\{\bA\in \O_q(\Lambda) \ \big| \ \Ja(\ba) = q^{|\bA|}\right\}\right| = |\O_q(\Lambda_0)|,\] as claimed. 
Now, by definition of $P^\ast_\Lambda(T)$ and by construction of the decomposition of $\Lambda$, it follows that 
\begin{align*}
P^\ast_\Lambda(T)
&= \frac{P(\Lambda, T)}{(1-qT)^\rho}
=\prod_{\bA\in\O_q(\Lambda_0)} \frac{1-\Ja(\ba) T^{|\bA|}}{1-qT} \cdot \prod_{\bA\in\O_q(\Lambda^\ast\sqcup \Lambda')} \left(1-\Ja(\ba) T^{|\bA|}\right) \\ 
&=\prod_{\bA\in\O_q(\Lambda_0)} \frac{1- (qT)^{|\bA|}}{1-qT} \cdot 
\prod_{\bA\in\O_q(\Lambda^\ast)} \left(1-\Ja(\ba) T^{|\bA|}\right). 
\end{align*}
Evaluating this expression at $T=q^{-1}$ yields the result. 
\ProofEnd \end{proof}

In the following lemma, we record a few useful facts about the action of $q$ on subsets of $G_d$.

\begin{lemm}\label{lemm.prel.actq}
 Let $\Lambda\subset G_d$ be a nonempty subset which is stable under the action of $(\Z/d\Z)^\times$ by multiplication. 
 Then the following upper bounds hold:
\begin{enumerate}[(i)]
\item \label{item.i.actq} 
$ \displaystyle \sum_{\bA\in\O_q(\Lambda)}  |\bA| = |\Lambda| \leq |\Lambda|$,
\item \label{item.ii.actq} 
$ \displaystyle \sum_{\bA\in\O_q(\Lambda)}  1 =  |\O_q(\Lambda)| \ll \log q\cdot \frac{|\Lambda|}{\log |\Lambda|}$,
\item \label{item.iii.actq} 
$ \displaystyle \sum_{\bA\in\O_q(\Lambda)}  \log |\bA|  \ll \log q\cdot \frac{|\Lambda|\cdot \log\log |\Lambda|}{\log |\Lambda|}$.
\end{enumerate}
All the involved constants are absolute and effective.
\end{lemm}

\begin{proof} The first assertion follows directly from the fact that the set $\O_q(\Lambda)$ can be written as a disjoint union of orbits under the action $q$.
To prove part \eqref{item.ii.actq}, we introduce the following notation: for any divisor $d'$ of $d$, we let ${\Lambda_{d'}:= \left\{\ba=(a_0, \dots, a_3)\in\Lambda \ \big| \ \gcd(d, a_0, \dots, a_3)=d/d' \right\}}$. 
It is clear that $\Lambda = \bigsqcup_{d'\mid d} \Lambda_{d'}$ and that the action of $q$ on $\Lambda$ can be restricted to each $\Lambda_{d'}$, yielding the decomposition $\O_q(\Lambda) = \bigsqcup_{d'\mid d} \O_q(\Lambda_{d'})$. We remark that $\Lambda_{1} \subset \{(0,0,0,0)\}$ and $|\O_q(\Lambda_1)|\leq1$: in what follows, we thus concentrate on divisors $d'\geq 2$ of $d$.
Any orbit $|\bA|$ containing an element $\ba\in\Lambda_{d'}$  has length $|\bA| = |\sbgp{q}{d'}| = o_q(d')$: 
this implies that $|\O_q(\Lambda_{d'})| = |\Lambda_{d'}|/o_q(d')$. We also note that, by definition of $o_q(d')$, $d'$ divides $q^{o_q(d')}-1$ so that $\log d' \leq o_q(d') \cdot  \log q$. Putting these remarks together, we see that
\[ |\O_q(\Lambda)| 
= \sum_{d'\mid d} |\O_q(\Lambda_{d'})| 
\leq 1+ \sum_{\substack{d'\mid d \\ d'\geq 2}} \frac{|\Lambda_{d'}|}{o_q(d')}
\leq 1+\log q \cdot \sum_{\substack{d'\mid d\\ d'\geq 2}}  \frac{|\Lambda_{d'}|}{\log d'}.\]
Let $X\in [2,d]$ be a parameter, 
we split the last sum into two parts, which we will estimate separately:
\[ \sum_{\substack{d'\mid d\\ d'\geq 2}}  \frac{|\Lambda_{d'}|}{\log d'}
 = \sum_{\substack{d'\mid d\\ 2\leq d' \leq  X}}  \frac{|\Lambda_{d'}|}{\log d'}
 +
 \sum_{\substack{d'\mid d\\ d'>X}}  \frac{|\Lambda_{d'}|}{\log d'}.\]
To bound the first sum, we remark that for each $d'\mid d$ (with $d'\geq 2$), $\Lambda_{d'}$ is in one-to-one correspondence with a subset of $(\Z/d'\Z)^3$ so that $|\Lambda_{d'}| \leq {d'}^3$. Since the function $y\mapsto y^3/\log y$ is increasing on $[2,X]$, the first sum satisfies
\[ \sum_{\substack{d'\mid d\\ 2\leq d' \leq  X}}  \frac{|\Lambda_{d'}|}{\log d'}
\leq \sum_{\substack{d'\mid d\\ 2\leq d' \leq  X}}  \frac{{d'}^3}{\log d'}
\leq \frac{X^3}{\log X}\cdot \sum_{\substack{d'\mid d\\ 2\leq d' \leq  X}} 1
\leq \frac{X^3}{\log X}\cdot \sum_{{2\leq d' \leq  X}} 1 \leq \frac{X^4}{\log X}. \]
To treat the second sum, we use that $y\mapsto (\log y)^{-1} $ is decreasing on $[X, +\infty[$ and the decomposition ${\Lambda= \bigsqcup_{d'\mid d}\Lambda_{d'}}$:
\[ \sum_{\substack{d'\mid d\\ d' >  X}}  \frac{|\Lambda_{d'}|}{\log d'}
\leq \sum_{\substack{d'\mid d\\ 2\leq d' \leq  X}}  \frac{|\Lambda_{d'}|}{\log X}
\leq \frac{1}{\log X} \sum_{d'\mid d} |\Lambda_{d'}| 
\leq \frac{|\Lambda_d|}{\log X}. \]
Summing the two contributions and choosing $X=|\Lambda|^{1/4}$ leads to 
\[ |\O_q(\Lambda)| \leq 1+ \log q \cdot \frac{X^4 + |\Lambda|}{\log X}  \leq 1+ 8\log q\cdot \frac{|\Lambda|}{\log |\Lambda|} \ll \log q\cdot  \frac{|\Lambda|}{\log |\Lambda|}. \] 
This proves part \eqref{item.ii.actq} of the Lemma (with a hidden absolute constant $c_5\leq 9$). 
We finally turn to the proof of part \eqref{item.iii.actq}; again, we use the decomposition of $\O_q(\Lambda)$ as the disjoint union of $\O_q(\Lambda_{d'})$ (with $d'\mid d$) and the fact that the orbits of $\ba\in\Lambda_{d'}$ all have the same length $o_q(d')$: 
\[\sum_{\bA\in\O_q(\Lambda)} \log|\bA| 
= \sum_{\substack{d'\mid d \\ d'\geq 2}} \sum_{\bA \in \O_q(\Lambda_{d'})}\log |\bA|
= \sum_{\substack{d'\mid d \\ d'\geq 2}} |\O_q(\Lambda_{d'})|\cdot \log o_q(d')
= \sum_{\substack{d'\mid d \\ d'\geq 2}} \frac{|\Lambda_{d'}|}{o_q(d')}\cdot \log o_q(d').
\]
We introduce a new parameter $Y\in[3,d]$ and we split the sum into two parts according to the size of $d'$ with respect to $Y$. To bound the sum over ``small divisors'' $d'$ of $d$ (\ie{} $d'\mid d$ and $2\leq d'\leq Y$), we use that $\log o_q(d')\leq \log \phi(d')\leq \log Y$ and 
\[ \sum_{\substack{d'\mid d \\ 2\leq d'\leq Y}} \frac{|\Lambda_{d'}|}{o_q(d')}\cdot \log o_q(d')
\leq \log Y \sum_{\substack{d'\mid d \\ d'\geq 2}} \frac{|\Lambda_{d'}|}{o_q(d')} 
= \log Y \sum_{\substack{d'\mid d \\ d'\geq 2}} |\O_q(\Lambda_{d'})|
\leq 
|\O_q(\Lambda)| \cdot \log Y.
\]
The sum over ``big divisors'' of $d$ (those $d'\mid d$ such that $d'\geq Y$) can be bounded from above by remarking that $y\mapsto (\log y)/y$ is decreasing on $[3, +\infty[$: 
\[ \sum_{\substack{d'\mid d \\ d'> Y}} \frac{|\Lambda_{d'}|}{o_q(d')}\cdot \log o_q(d')
\leq \sum_{\substack{d'\mid d \\ d'> Y}} \frac{\log o_q(d')}{o_q(d')} \cdot |\Lambda_{d'}|
\leq  \frac{\log Y}{Y}  \sum_{\substack{d'\mid d \\ d'> Y}}|\Lambda_{d'}| 
\leq  \frac{\log Y}{Y}  \sum_{d'\mid d }|\Lambda_{d'}|
= \frac{\log Y}{Y}  |\Lambda|.
\]
From part \eqref{item.ii.actq} that we have just proved, we deduce that:
\[\sum_{\bA\in\O_q(\Lambda)} \log|\bA| 
=  \sum_{\substack{d'\mid d \\ d'\geq 2}} \frac{|\Lambda_{d'}|}{o_q(d')}\cdot \log o_q(d')
\leq |\O_q(\Lambda)| \cdot \log Y +  \frac{\log Y}{Y}\cdot |\Lambda|
\leq |\Lambda|\log Y \cdot\left( \frac{c_5 \log q}{\log|\Lambda|} + \frac{1}{Y}\right).
\]
On taking $Y= \log |\Lambda|/(c_5\log q)$, we get 
the desired inequality (with an absolute constant $c_6\leq 18$).
\ProofEnd \end{proof}

\section{An equidistribution statement}\label{sec.equidis}

\paragraph{} Let us temporarily turn to a more combinatorial problem 
and consider subsets of $\Z/d\Z$. The fractional part of $x\in\R$ will be denoted by $\partfrac{x}$. The map $m\in\Z/d\Z\mapsto \partfrac{{m}/d}$ 
allows us to view subsets of $\Z/d\Z$ as sequences in $[0,1]$: we may then study the distribution of subsets $H_d$ of $\Z/d\Z$ when $d$ grows. Of course, given a random set $H_d\subset \Z/d\Z$, there is no reason why $H_d$ should be particularly well-distributed in $\Z/d\Z$. 
Nonetheless, if $H_d$ is a subset of $(\Z/d\Z)^\times$ that is ``not too small'', we show that its ``translates'' $g\cdot H_d$ (by $g\in(\Z/d\Z)^\times$) are, on average, equidistributed in $\Z/d\Z$. More precisely:

\begin{theo} \label{theo.equidis}
Let $F:[0,1]\to\R$ be a function of bounded variation, and denote by $\mathscr{V}(F)$ its total variation.
Let $\mathscr{D}\subset\Z_{\geq 1}$ be an infinite set of positive integers.
For all $d\in\mathscr{D}$, suppose we are given a subset $H_d$ of $(\Z/d\Z)^\times$ and an element $a_d\in(\Z/d\Z)^\times$; assume that 
\[ |H_d| / \log\log d \xrightarrow[\substack{ d\in\mathscr{D}\\ d\to\infty}]{} +\infty. \]
Then, for all $\epsilon\in(0, 1/4)$, when $d\in\mathscr{D}$ goes to $+\infty$, one has 
\begin{equation}\label{eq.equidis}
\frac{1}{\phi(d)} \sum_{g\in(\Z/d\Z)^\times} 
\left| \int_{0}^1 F(t) \dd t  - \frac{1}{|H_d|} \sum_{h\in H_d} F\left(\partfrac{\frac{a_d\cdot gh}{d}}\right) \right|  
\leq c_7 \cdot \mathscr{V}(F) \cdot \left(\frac{\log\log d}{|H_d|}\right)^{1/4 - \epsilon}. 
\end{equation}
The constant $c_7>0$ is effective and depends at most on $\epsilon>0$. In particular, note that this upper bound is entirely independent of the choice of $a_d$. 
\end{theo}

This equidistribution theorem constitutes the main analytic input in the proof of our lower bound on special values (see Theorem \ref{theo.lowerbnd}), and may be of interest for other applications. See \cite[Lemma 7.10]{HP15} for a related statement.
This result applies in particular to $\mathcal{C}^1$ functions $F:[0,1]\to\R$, in which case $\mathscr{V}(F)$ is simply $\int_0^1 |F'(t)|\dd t$. 

 Note that the outer average on $(\Z/d\Z)^\times$ is necessary, as exemplified by the case where $d_N=q^N-1$ ($q\geq 2$ a fixed integer, $N\in\Z_{\geq 1}$), $a_{d_N}=1$ and $H_{d_N}= \{1, q, q^2, \dots, q^{N-1}\} \subset (\Z/d_N\Z)^\times$, for which the term corresponding to $g=1\in(\Z/d_N\Z)^\times$ in \eqref{eq.equidis} is not necessarily small. I thank Igor Shparlinski for pointing out the necessity of the coprimality conditions ($a_d\in(\Z/d\Z)^\times$ and $H_d\subset (\Z/d\Z)^\times$), without which equidistribution may fail.
The proof of Theorem \ref{theo.equidis} will occupy the rest of the present section.

\subsection[Fourier transform]{Fourier transform on $\Z/d\Z$}\label{sec.fourier}
Fix a positive integer $d\geq 2$, any 
function $\psi:\Z/d\Z\to\C$ has a Fourier transform  ${\hhat{\psi}:\Z/d\Z\to\C}$ defined by
\[ \forall y\in\Z/d\Z, \qquad \hhat{\psi}(y):= \sum_{x\in\Z/d\Z} \psi(x)\cdot \ee\left(\tfrac{xy}{d}\right),\]
where   $\ee(x):= \e^{2i\pi x}$. 
In this context, an analogue of \emph{Plancherel's equality} can be seen to hold: 
\begin{equation}\label{eq.parseval}
\sum_{y\in\Z/d\Z} |\hhat{\psi}(y)|^2 = d\cdot \sum_{x\in\Z/d\Z} |\psi(x)|^2,
\end{equation}
or, with more evocative notations, $\|\hhat{\psi}\|_2^2 = d\cdot \|\psi\|_2^2$.

Let $H\subset \Z/d\Z$ be a subset and denote by $\ind:\Z/d\Z\to\R$ the characteristic function of $H$. Clearly, we have 
$\|\hhat{\ind}\|_\infty =\hhat{\ind}(0) = |H|$, so that $\psi_H:= \ind(.)/|H|$ satisfies $\|\hhat{\psi_H}\|_\infty = 1$ (where $\|\psi'\|_\infty:= \max_x |\psi'(x)|$ denotes the sup-norm of a function $\psi'$). 
Let us prove the following lemma: 
 
\begin{lemm}\label{lemm.fourier.lemma}
 For any $\beta\in(0,1]$ and any $k\in\Z$ such that $k\nequiv 0\bmod{d}$, one has
 \begin{equation}\label{eq.fourier.coro}
\frac{1}{\phi(d)}\sum_{g\in(\Z/d\Z)^\times}\big|\hhat{\psi_H}(kg)\big| 
\leq \beta + \e^\gamma \cdot \frac{\log\log d}{|H|}\cdot \frac{\gcd(k,d)}{\beta^2},
 \end{equation}
where 
$\gamma\simeq 0.577...$ denotes the Euler-Mascheroni constant.
\end{lemm}

\begin{proof} Let us cut the sum 
into two parts according to the given  parameter $\beta\in(0,1]$: 
\[ \sum_{g\in(\Z/d\Z)^\times} |\hhat{\psi_H}(k g)| 
= \sum_{g  \text{ s.t. }   |\hhat{\psi_H}(k g)| \leq \beta}  |\hhat{\psi_H}(k g)|
 + 
 \sum_{g \text{ s.t. } |\hhat{\psi_H}(k g)| > \beta}
 |\hhat{\psi_H}(k g)|.
\] 
The first sum is clearly $\leq \beta\cdot |(\Z/d\Z)^\times| = \beta \cdot \phi(d)$. 
Since $\|\hhat{\psi_H}(y)\|_\infty\leq 1$, the second sum is bounded by:
\[ 
\sum_{\substack{g \in(\Z/d\Z)^\times \text{ s.t.} \\  |\hhat{\psi_H}(k g)| >\beta }} 
|\hhat{\psi_H}(kg)|
\leq 1 \cdot \Big|\big\{ g\in (\Z/d\Z)^\times \ \big| \  |\hhat{\psi_H}(kg)| > \beta \big\}\Big|.
\]
For any $y\in\Z/d\Z$, there are at most $\gcd(k,d)$ elements $g\in(\Z/d\Z)^\times$ such that $kg=y$, so that
\[
\Big|\big\{ g\in (\Z/d\Z)^\times \ \big| \  |\hhat{\psi_H}(kg)| > \beta \big\}\Big|
\leq \gcd(k,d)\cdot
\Big|\big\{ y\in \Z/d\Z \ \big| \  |\hhat{\psi_H}(y)| > \beta \big\}\Big|.\]
By Plancherel's equality \eqref{eq.parseval}, we have
\[ \beta^2 \cdot \Big|\big\{ y\in \Z/d\Z \ \big| \  |\hhat{\psi_H}(y)| > \beta \big\}\Big|
\leq \|\hhat{\psi_H}\|_2^2
= d\cdot \|\psi_H\|_2^2
\leq d\cdot|H|\cdot\|\psi_H\|_\infty^2 \leq \frac{d}{|H|}.
\]
From which it follows that 
\[ \sum_{g \text{ s.t. } |\hhat{\psi_H}(k g)| > \beta}
 |\hhat{\psi_H}(k g)| 
 \leq \frac{d}{|H|}\cdot \frac{\gcd(k,d)}{\beta^2}. \]
Adding the two contributions, 
we arrive at:
 \[ \frac{1}{\phi(d)} \sum_{g\in(\Z/d\Z)^\times} |\hhat{\psi_H}(kg)|
\leq \beta + \frac{d}{\phi(d)\cdot|H|}\cdot \frac{\gcd(k,d)}{\beta^2}.
 \]
We conclude the proof 
by  using the classical theorem stating that  $d/\phi(d) \leq \e^{\gamma} \cdot \log\log d$ (see \cite[Thm. 328]{HardyWright}) where $\gamma$ denotes the Euler-Mascheroni constant. 
 \ProofEnd \end{proof}

\subsection{Tools from equidistribution theory} Before passing to the proof of Theorem \ref{theo.equidis}, we 
recall the following two results about distribution of sequences in $[0,1]$. The reader may consult \cite[Chap. 2]{KuNi} for more details and for proofs.  For any finite sequence $x_1, x_2, \dots, x_N$ of $N$ points in $[0,1]$, we define its \emph{discrepancy} by
\[D((x_n)_{n=1}^N):= \sup_{I\subset [0,1]}  \left| \mu(I) - \frac{1}{N} \left|\big\{ n\in\intent{1,N} \ \big| \ x_n\in I\big\}\right|\right|, \]
the supremum being taken over all intervals $I \subset [0,1]$, whose length is denoted by $\mu(I)$. 
The inequality below is sometimes called a quantitative version of Weyl's criterion for equidistribution of a sequence in~$[0,1]$: 

\begin{prop}[Inequality of Erdös and Turán]\label{prop.erdosturan}
  Let $x_1, \dots, x_N$ be a sequence of $N$ points in $[0,1]$. Then, for all integers $K\geq 1$, the discrepancy $D((x_n)_{n=1}^N)$ is bounded by
\begin{equation}\label{eq.erdosturan}
D((x_n)_{n=1}^N) \leq \frac{6}{K+1} + \frac{4}{\pi} \sum_{k=1}^K\frac{1}{k}\left| \frac{1}{N}\sum_{n=1}^N \e^{2i\pi \cdot k \cdot x_n}\right|.
\end{equation}
\end{prop}
The next proposition comes from the world of numerical integration and roughly states that the average of a well-behaved function on $[0,1]$ can be approximated by averaging its values at finitely many points, provided that these points are sufficiently well-distributed. 

\begin{prop}[Koksma's inequality] \label{prop.koksma}
Let $F:[0,1]\to\R$ be a function of bounded variation,
and $x_1, \dots, x_N$ be a sequence of $N$ points in $[0,1]$ with discrepancy $D((x_n)_{n=1}^N)$. Then, one has
\begin{equation}\label{eq.koksma}
\left| \int_0^1 F(t) \dd t - \frac{1}{N} \sum_{n=1}^N F(x_n)\right| \leq \mathscr{V}(F) \cdot D((x_n)_{n=1}^N),
\end{equation}
where $\mathscr{V}(F)$ denotes the total variation of $F$.
\end{prop}

\subsection{Proof of Theorem \ref{theo.equidis}} 
 Let hypotheses be as in the statement of Theorem \ref{theo.equidis}. For any $d\in\mathscr{D}$, we let
\[\Theta_d(a_d, H_d):=\frac{1}{\phi(d)} \sum_{g\in(\Z/d\Z)^\times} 
\left| \int_{0}^1 F(t) \dd t  - \frac{1}{|H_d|} \sum_{h\in H_d} F\left(\partfrac{\frac{a_d\cdot gh}{d}}\right) \right|.\]
If $F$ is constant, there's nothing to prove, for then $\Theta_d(a_d, H_d)=0$. 
If not, dividing throughout 
by $\mathscr{V}(F)\neq 0$, we may assume that $F$ has total variation $\mathscr{V}(F)=1$. 
Since $a_d\in(\Z/d\Z)^\times$, one sees that $\Theta_d(a_d, H_d)=\Theta_d(1,H_d)$ by reindexing the outer sum; thus, we need only prove
that $\Theta_d(1, H_d)$ satisfies the desired bound.
Whenever it is necessary, we will assume that $d\in\mathscr{D}$ is big enough that $\log\log d$ is positive. 

For any $g\in(\Z/d\Z)^\times$, let $(x_h)_{h\in H_d}$ be the finite sequence 
defined by ${x_h = \partfrac{gh/d} \in [0,1]}$ for all $h\in H_d$. 
We successively use Koksma's inequality (Proposition \ref{prop.koksma}) and the Erdös-Turán inequality (Proposition \ref{prop.erdosturan}) on $(x_h)_{h\in H_d}$ to get 
\begin{align}\label{eq.equidis.pf1}
\left| \int_{0}^1 F(t) \dd t  - \frac{1}{|H_d|} \sum_{h\in H_d} F\left(\partfrac{\frac{gh}{d}}\right) \right|  
&\leq \mathscr{V}(F) \cdot D((x_h)_{h\in H_d}) = D((x_h)_{h\in H_d}) \notag \\ 
&\leq 
\frac{6}{K+1} + \frac{4}{\pi} \sum_{k=1}^K\frac{1}{k}\left| \frac{1}{|H_d|}\sum_{h\in H_d}\ee\left( k \cdot \partfrac{\frac{gh}{d}}\right) \right|,
\end{align}
for all integers $K\geq 1$. 
Denote by $\ind: \Z/d\Z \to\R$ the characteristic function of $H_d\subset (\Z/d\Z)^\times$ seen as a subset of $\Z/d\Z$, and
set  $\psi_d:= \ind(.) / |H_d|$ as in subsection \ref{sec.fourier}.
Since $y\mapsto \partfrac{y}$ is a $1$-periodic function, 
the inner sums in \eqref{eq.equidis.pf1} are easily written as Fourier coefficients of $\psi_d$:
\[
\forall k\geq 1, \quad   
\frac{1}{|H_d|}\sum_{h\in H_d}\ee\left( k \cdot \partfrac{\frac{gh}{d}}\right)
= \frac{1}{|H_d|}\sum_{h'\in \Z/d\Z} \ind(h') \cdot \ee\left(\frac{k gh'}{d} \right) 
= \hhat{\psi_d}(k g).
\]
Using this expression and averaging inequalities \eqref{eq.equidis.pf1} over $g\in(\Z/d\Z)^\times$ yields 
\begin{equation}\label{eq.equidis.pf1bis} 
0
\leq \Theta_d(1, H_d)
\leq \frac{6}{K+1}  + \frac{4}{\pi} \sum_{k=1}^K\frac{1}{k}  \cdot \left(\frac{1}{\phi(d)}\sum_{g\in(\Z/d\Z)^\times} |\hhat{\psi_d}(k  g)| \right).
\end{equation}
We now use Lemma \ref{lemm.fourier.lemma} with $k\neq 0$ and plug \eqref{eq.fourier.coro} in the last displayed inequality: for all $\beta\in(0,1]$,
\begin{align*}
\sum_{k=1}^K\frac{1}{k}  \cdot \left(\frac{1}{\phi(d)}\sum_{g\in(\Z/d\Z)^\times} |\hhat{\psi_d}(-k  g)| \right) 
 &\leq   \sum_{k=1}^K\frac{1}{k}  \cdot \left(\beta + \e^\gamma \cdot \frac{\log\log d}{|H_d|}\cdot \frac{\gcd(k,d)}{\beta^2}   \right)  \\
&\leq \beta\cdot \log(K+1)+ \e^\gamma \cdot \frac{\log\log d}{|H_d|} \cdot \frac{(K+1)}{\beta^2},
\end{align*}
 because $\gcd(k,d)/k \leq 1$ for all $k\geq 1$.
Setting $K'=K+1$ and $f(d):= (\log\log d)/|H_d|$, we obtain that: 
\[
\Theta_d(1, H_d)
\leq \frac{6}{K'} +  \frac{4}{\pi} \cdot \beta \cdot \log K'  + \frac{4\e^{\gamma} f(d)}{\pi} \cdot \frac{K'}{\beta^2}.
\]
We choose $K' = K+1= \partint{\e^{1/4-\epsilon}\cdot f(d)^{-1/4+\epsilon}}$ and  $\beta = (3\pi/2)\cdot(K' \log K')^{-1}$ (recall that, by assumption, $f(d)\to 0^+$ when $d\to\infty$) and we obtain that
\[
\Theta_d(1, H_d)
\leq 12\cdot f(d)^{1/4-\epsilon} + \frac{\e^{\gamma}}{9\pi^3}\cdot f(d)^{1/4+3\epsilon} \cdot (1 - \log f(d))^2. 
\]
Noticing that  one has $0 \leq 1-\log y\leq (2\epsilon)^{-1}\cdot y^{-2\epsilon}$ for all $y\in(0, +\infty)$, 
straightforward manipulations imply 
the upper bound that was announced in Theorem \ref{theo.equidis}:
\[
\Theta_d(1, H_d)
\leq \left(12 + \frac{\e^{\gamma}}{36\pi^3} \cdot \epsilon^{-2}\right)\cdot f(d)^{1/4-\epsilon},
\]
with an explicit expression for the constant $c_7$. 
\hfill $\square$

\section{Upper bounds} 
\label{sec.upperbound}

\subsection[Upper bound]{Upper bound on $P^\ast(\Lambda)$}

We now come back to studying the size of  special values $P^\ast(\Lambda)$ as defined in section \ref{sec.setting}. The main goal of this section is to prove an upper bound on $P^\ast(\Lambda)$. 

\begin{theo} \label{theo.upperbnd}
Let $\F_q$ be a finite field and $d\geq 2$ be an integer coprime with~$q$. Given a nonempty $(\Z/d\Z)^\times$-stable subset $\Lambda \subset G_d$, consider the polynomial ${P(\Lambda, T)\in\Z[T]}$ associated to $\Lambda$ as in Definition \ref{defi.polyn.lambda}. 
The special value $P^\ast(\Lambda)$ of $P(\Lambda, T)$ satisfies 
\begin{equation}\label{eq.upperbnd}
\log |P^\ast(\Lambda)| \ll \log q \cdot |\Lambda| \cdot \frac{\log\log |\Lambda|}{\log |\Lambda|}.
\end{equation}
Here, the implicit constant is absolute and effective. 
\end{theo}

\begin{proof} 
In Lemma \ref{lemm.expr.spval}, we have given an explicit expression of $P^\ast(\Lambda)$, of which we take the logarithm. 
We then use the triangle inequality, noting that $|\Ja(\ba)|=q^{|\bA|}$  when $\ba\in\Lambda^\ast$:
\begin{align*}
\log |P^\ast(\Lambda)|
&= \sum_{\bA\in\O_q(\Lambda_0)} \log |\bA|
+\sum_{\bA\in\O_q(\Lambda^\ast)} \log \left|1-\frac{\Ja(\ba)}{q^{|\bA|}}\right| 
\leq \sum_{\bA\in\O_q(\Lambda)} \log |\bA|
+\sum_{\bA\in\O_q(\Lambda^\ast)} \log \left(1+\frac{|\Ja(\ba)|}{q^{|\bA|}}\right) 
\\
&\leq \sum_{\bA\in\O_q(\Lambda)} \log |\bA|
+\log 2 \cdot |\O_q(\Lambda^\ast) |
\leq \sum_{\bA\in\O_q(\Lambda)} \log |\bA|
+\log 2 \cdot |\O_q(\Lambda) |.
\end{align*}
Using items \eqref{item.ii.actq} and \eqref{item.iii.actq} in Lemma \ref{lemm.prel.actq} leads to:  
\[\log|P^\ast(\Lambda)|
\leq c_6\log q \cdot  \frac{|\Lambda| \cdot\log\log|\Lambda|}{\log|\Lambda|} + c_5\log 2\cdot \frac{ \log q \cdot |\Lambda|}{\log|\Lambda|}.\] 
This upper bound on $\log|P^\ast(\Lambda)|$ is easily seen to be stronger than the desired one. For completeness, we note that crude bounds lead to $\log|P^\ast(\Lambda)|\leq C_2\cdot \log q \cdot  \frac{|\Lambda|\cdot\log\log|\Lambda|}{\log |\Lambda|}$, where $C_2 \leq 25 $. \ProofEnd \end{proof}

\subsection{Upper bound on the `rank''}
In passing, we record the following upper bound on the order of vanishing of $P(\Lambda, T)$ at $T=q^{-1}$, which  we translate into an upper bound on the rank of the Néron-Severi groups of Fermat surfaces.

\begin{prop} \label{theo.rank.bound}
Given a nonempty subset $\Lambda\subset G_d$ which is stable under multiplication by $(\Z/d\Z)^\times$, 
the multiplicity of $T=q^{-1}$  as a root of $P(\Lambda, T)$ 
satisfies 
\begin{equation}
\ord_{T=q^{-1}} P(\Lambda, T) \ll  \log q\cdot \frac{|\Lambda| }{\log |\Lambda|}.
\end{equation}
Here too, the implicit constant is absolute and effective.
\end{prop}
\begin{proof}Recall from Lemma \ref{lemm.expr.spval} that $\ord_{T=q^{-1}}P(\Lambda, T) =|\O_q(\Lambda_0)|$, with $\Lambda_0= \left\{\ba\in\Lambda \ \big| \ \Ja_{\ba} = q^{|\bA|}\right\}$ and note that $|\O_q(\Lambda_0)| \leq |\O_q(\Lambda)|\leq c_5 \log q \cdot \frac{|\Lambda|}{\log |\Lambda|}$ by item \eqref{item.ii.actq} in Lemma \ref{lemm.prel.actq}.
\ProofEnd \end{proof}

This proposition implies an upper bound on the Picard number $\rho(\ferm_d/\F_q)$ of Fermat surfaces.
\begin{coro} \label{coro.rank.bound}
For any finite field $\F_q$ and any integer $d\geq 2$ coprime with $q$, let $\ferm_d$ be the Fermat surface surface over $\F_q$. The  rank $\rho(\ferm_d/\F_q)$ of the Néron-Severi group $\NS(\ferm_d)$ 
 satisfies 
 \begin{equation}\label{eq.rank.bound}
\rho(\ferm_d/\F_q) \ll \frac{\log q  \cdot d^3 }{\log d},
\end{equation}
for some absolute (small) constant.
\end{coro}

\begin{proof} For the duration of the proof, we choose $\Lambda_\ferm = G_d$ (note that $|\Lambda_\ferm| = d^3$). As we've seen in Example \ref{ex.fermat.P2}, one has $P(\Lambda_\ferm, T) = P_2(\ferm_d/\F_q, T)$.   Part (1) of the Artin-Tate conjecture \ref{conj.AT} implies that 
$\rho(\ferm_d/\F_q)$ equals the ``analytic rank'' $\ord_{T=q^{-1}}P_2(\ferm_d/\F_q, T)$. The corollary is then a direct consequence of Proposition \ref{theo.rank.bound} above (the implicit absolute constant $c_3$ can be chosen $c_3\leq 3$). 
\ProofEnd \end{proof}

One may compare the bound \eqref{eq.rank.bound} to the ``geometric'' bound $\rho(\ferm_d/\F_q) \leq \dim \H^2(\ferm_d)$ (see \cite{Igusa}), \ie{}
\[0\leq \rho(\ferm_d/\F_q) \leq \rho(\bar{\ferm_d}/\bar{\F_q})
\leq \dim \H^2(\ferm_d) = (d-1)(d^2-3d+3)\leq (d-1)^3.\]
This latter bound does not ``see'' the ``growth'' of the rank of the Néron-Severi groups $\NS(\ferm_d\times_{k} k')$ in a tower of finite extensions $k'/k$ of $k=\F_q$. 
We can also prove that \eqref{eq.rank.bound} is asymptotically optimal:

\begin{prop}\label{prop.rankbound.optimal} For any finite field $\F_q$, 
 there are infinitely many integers $d'$, prime to $q$, 
such that 
\[\rho(\ferm_{d'}/\F_q)\gg_q \frac{ {d'}^3 }{\log d'},\]
where the implied constant is effective and depends only on $q$. 
\end{prop}

\begin{proof}For any integer $n\geq1$, put $d_n = q^n+1$. 
For such $d_n$, it is possible to prove that $o_q(d_n) = 2n$ (see \cite[Lemma 8.2]{Ulmer_LargeRk}). Moreover, a theorem of Shafarevich and Tate shows that $\Ja(\ba) = q^{|\bA|}$ for all $\ba \in G_{d_n}^\circ$ (see \cite{ShaTate_Rk} or \cite[Prop. 8.1]{Ulmer_LargeRk}). Combining this with part (1) of the Artin-Tate Conjecture \ref{conj.AT} and Lemma \ref{lemm.expr.spval}, we see that 
\[\rho(\ferm_{d_n}/\F_q) = \ord_{T=q^{-1}}P_2(\ferm_{d_n}/\F_q, T) = |\O_q(G_{d_n})|. \]
On the other hand, with an argument similar to the proof of Lemma \ref{lemm.prel.actq}, it is possible to bound $|\O_q(G_{d_n})|$ from below. For any divisor $d'$ of $d_n$, let $H_{d'} = \left\{\ba=(a_0, \dots, a_3) \in G_{d_n}\ \big| \ \gcd(d,a_0, \dots, a_3) = d_n/d'\right\}$. 
Then 
\[ |\O_q(G_{d_n})| 
\geq \sum_{\substack{d'\mid d \\ d'\geq 2}} \frac{|H_{d'}|}{o_q(d')} 
\geq \frac{1}{o_q(d_n)} \cdot \sum_{\substack{d'\mid d \\ d'\geq 2}} |H_{d'}| 
=\frac{|G_{d_n}|-1}{o_q(d_n)}
= \frac{d_n^3-1}{2n}
\gg_q \frac{d_n^3}{\log d_n}.
\]
This proves that the integers $(d_n)_{n\geq 1}$ satisfy the statement of the Proposition. 
\ProofEnd \end{proof}

\section[Lower bound(s)]{Lower bound(s) on $P^\ast(\Lambda)$}\label{sec.lowerbound}

This section is devoted to proving 
lower bounds on $P^\ast(\Lambda)$.  For the rest of this section, we fix a finite field $\F_q$ of characteristic $p$ and an integer $d\geq 2$ coprime with $q$. Let us start by giving a ``trivial'' lower bound on~$P^\ast(\Lambda)$, which could be called a ``Liouville type'' lower bound:

\begin{prop}\label{prop.triv.lowerbnd} 
Let $d\geq 2$ be an integer coprime with $q$, and $\Lambda$ be a nonempty $(\Z/d\Z)^\times$-stable subset of $G_d$.
The special value $P^\ast(\Lambda)$ of the polynomial $P(\Lambda, T)$ associated to $\Lambda$ satisfies: 
\begin{equation}\label{eq.triv.lowerbnd}
\log |P^\ast(\Lambda)| \geq -\log q \cdot |\Lambda|
\qquad \text{\ie{} }\quad 
|P^\ast(\Lambda)|\geq \frac{1}{q^{|\Lambda|}}.
\end{equation}
\end{prop}

\begin{proof} By definition, $P^\ast(\Lambda)$ is the value at $T=q^{-1}$ of a polynomial with integral coefficients, namely $P^\ast_{\Lambda}(T)$. 
Thus, it is obvious that $|P^\ast_\Lambda(q^{-1})| \in \Z[q^{-1}]$ is of the form $N/q^{\deg P^\ast_\Lambda(T)}$ for some integer $N\geq 0$ prime to $q$. 
Since $P^\ast_\Lambda(T)$ does not vanish at $T=q^{-1}$, the numerator $N$ is actually positive. Hence $N\geq 1$ and the lower bound follows from the simple observation that $\deg P^\ast_{\Lambda}(T)\leq \deg P(\Lambda, T)\leq |\Lambda|$.
\ProofEnd \end{proof}

However, the lower bound \eqref{eq.triv.lowerbnd} is sometimes far from the truth.  Indeed, assume that $d$ divides $q^n+1$ for some $n\in\Z_{\geq 1}$, 
then the theorem of Shafarevich and Tate  mentioned above (see \cite{ShaTate_Rk})  tells us that $\Ja(\ba)=q^{|\bA|}$ for all $\ba\in G_d^\circ$. Therefore, in this case, for any $\Lambda\subset G_d$ as above, the set $\Lambda^\ast$ is empty and, by Lemma \ref{lemm.expr.spval}, the special value $P^\ast(\Lambda)$ is a positive integer:
\[\log |P^\ast(\Lambda)|  = \log \prod_{\Lambda\in\O_q(\Lambda_0) }|\bA| \geq0.\]
We now set out to prove a lower bound on $P^\ast(\Lambda)$ refining \eqref{eq.triv.lowerbnd} in a more  general case: 
\begin{theo} \label{theo.lowerbnd} 
Let $d\geq 2$ be an integer coprime with $q$, and $\Lambda$ be a nonempty $(\Z/d\Z)^\times$-stable subset of $G_d$. Assume that, for all $\epsilon\in(0,1/4)$, there exists $u\in(0,1)$ such that
\begin{equation}\label{eq.hyp}\tag{H}
\frac{1}{|\Lambda|}\cdot \left| \left\{\ba=(a_0, \dots, a_3)\in\Lambda \ \big| \ d>\max_i\{\gcd(d,a_i)\} > d^u\right\}\right| \leq c' \cdot \left(\frac{\log\log d}{\log d}\right)^{1/4-\epsilon},
\end{equation}
for some constant $c'$.
Then, for all $\epsilon\in(0,1/4)$, 
the special value $P^\ast(\Lambda)$ of the polynomial $P(\Lambda, T)$ 
satisfies 
\begin{equation}\label{eq.lowerbnd}
\log |P^\ast(\Lambda)| \gg - \log q \cdot |\Lambda| \cdot \left(\frac{\log\log  d}{\log d}\right)^{1/4-\epsilon},
\end{equation}
the implicit (positive) constant being effective (and depending at most on $p$, $\epsilon$, $u$ and $c'$). 
\end{theo}




The rest of this section is devoted to the proofs of Theorem \ref{theo.lowerbnd} and of Corollary \ref{theo.BNDS.GEN}.

\subsection{Lower bound on products}

Our first step towards Theorem \ref{theo.lowerbnd} will be to prove: 

\begin{theo}\label{theo.proto.lowerbnd}
Let $d\geq 2$ be an integer coprime with $q$, and $\Lambda$ be any nonempty $(\Z/d\Z)^\times$-stable subset of $G_d$. The special value $P^\ast(\Lambda)$ satisfies 
\begin{equation}\label{eq.proto.lowerbnd}
\log |P^\ast(\Lambda)| \geq -\log q \cdot w_p(\Lambda, d) \qquad \text{\ie{} }\quad 
|P^\ast(\Lambda)|\geq \frac{1}{q^{w_p(\Lambda, d)}},
\end{equation}
where $w_p(\Lambda, d) $ is given by $w_p(\Lambda, d):= \sum_{\ba\in\Lambda^\circ} w_p(\ba,d)$ with $\Lambda^\circ= \Lambda\cap G_d^\circ$ 
and
  \[ 
  w_p(\ba, d):=  \frac{1}{\phi(d)} \sum_{g\in(\Z/d\Z)^\times}  \max\left\{ 0, 
 \sum_{i=0}^3\left(  - \frac{1}{2} + \frac{1}{|\sbgp{p}{d}|}\sum_{\pi\in\sbgp{p}{d}} 
\partfrac{\frac{a_i  \pi g}{d}} \right) 
 \right\}. \]
\end{theo}

We note the (surprising?) fact that $w_p(\Lambda,d)$ does not depend on $q$, 
but only on $p$.
Of course, the usefulness of Theorem \ref{theo.proto.lowerbnd} relies on our ability to find a good upper bound on $w_p(\Lambda, d)$: indeed,
the immediate upper bound $w_p(\Lambda, d)\leq |\Lambda|$ only yields the trivial lower bound \eqref{eq.triv.lowerbnd} of Proposition \ref{prop.triv.lowerbnd}. 
This question is addressed the next subsection, where we prove that $w_p(\Lambda, d)$ is $o(|\Lambda|)$ (as $|\Lambda|\to\infty$), under the assumption of hypothesis \eqref{eq.hyp}.


\begin{proof} As we have seen (in Lemma \ref{lemm.expr.spval}), the special value $P^\ast(\Lambda)$ associated to $\Lambda$ has the following shape: 
\[ 
P^\ast(\Lambda) 
= \prod_{\bA\in\O_q(\Lambda_0)} |\bA| \cdot \prod_{\bA\in\O_q(\Lambda^\ast)} \left(1-\frac{\Ja(\ba)}{q^{|\bA|}}\right)
= (\text{integer})\cdot \prod_{\bA\in\O_q(\Lambda^\ast)} \left(1-\frac{\Ja(\ba)}{q^{|\bA|}}\right).\]
From this, we deduce that 
\[\log|P^\ast(\Lambda)| 
= \log|\text{integer}| + \log\prod_{\bA\in\O_q(\Lambda^\ast)} \left|1-\frac{\Ja(\ba)}{q^{|\bA|}}\right|
\geq \log\prod_{\bA\in\O_q(\Lambda^\ast)} \left|1-\frac{\Ja(\ba)}{q^{|\bA|}}\right|,  \]
and we shall now find lower bounds on the product  $\Pi^\ast_\Lambda:= \prod_{\bA\in\O_q(\Lambda^\ast)} \left(1-\Ja(\ba)q^{-|\bA|}\right)$ on the right-hand side. This product $\Pi^\ast_\Lambda$ is a nonzero element of $\Z[q^{-1}]$ and we need to bound from above the exponent of $q$ appearing in its denominator. 
We split the proof of Theorem \ref{theo.proto.lowerbnd} into three parts.

First, we recall the following proposition of Shioda (see \cite[Prop. 2.1]{Shioda_jacobi}) concerning Jacobi sums. For the convenience of the reader, we give a proof in our notations below.

\begin{lemm}[Shioda]\label{lemm.norm.jacobi}
 Let $\F_q$ be a finite field, 
 $d\geq 2$ be an integer coprime with $q$,
  and $\ba\in G_d^\circ\subset G_d$ (so that $|\Ja(\ba)| =q^{|\bA|}$). 
Set $v_\bA=\ord_p (q^{|\bA|}) = [\F_{q^{|\bA|}}: \F_p]$.
Let $\gp$ be the prime ideal of $K:=\Q(\zeta_d)$ that lies under ${\gP}$ (see section   \ref{sec.notations}) and denote by $\ord_\gp: K^\ast\to\Z$ the $\gp$-adic valuation. 
Then either $\Ja(\ba)=q^{|A|}$, or 
\begin{equation}\label{eq.norm.jacobi}
\log\norm_{K/\Q}\left(1-\frac{\Ja(\ba)}{q^{|\bA|}}\right)
\geq - (\log q^{|\bA|})\cdot \sum_{g\in(\Z/d\Z)^\times} \max\left\{ 0, 1- \frac{\ord_{\gp} \Ja(g\cdot \ba)}{ v_{\bA} }\right\}.
\end{equation}
\end{lemm}

\begin{proof} Fix a set $g_1=1, g_2, \dots, g_s \in (\Z/d\Z)^\times$ of coset representatives of $\sbgp{p}{d}$ in $(\Z/d\Z)^\times$ (hence $s=\phi(d)/o_p(d)$, where $o_p(d)=|\sbgp{p}{d}|$ is the multiplicative order of $p\bmod{d}$). For any $g\in(\Z/d\Z)^\times$, we denote by $\sigma_g\in\Gal(K/\Q)$ the corresponding automorphism of $K$ in the usual isomorphism. Now let $\gp_i:= (\sigma_{g_i})^{-1}(\gp)$ for $i=1, \dots, s$: it is well-known that $p$ decomposes as $p\cdot \Z[\zeta_d] = \gp_1\cdot \dotso\cdot \gp_s$ in $\Z[\zeta_d]$ (see \cite[Chap. 13]{IR}) and that $\norm\gp_i = \norm \gp = p^{o_p(d)}$ for all $i$. Note that $\ord_\gp q^{|\bA|} = v_\bA$.

It is then clear that $q^{|\bA|}\cdot\Z[\zeta_d] =  \prod_{i=1}^s \gp_i^{v_\bA}$. 
Since $|\Ja(\ba)|=q^{|\bA|} = p^{v_\bA}$, the integral ideal of $\Z[\zeta_d]$ generated by $\Ja(\ba)$ is concentrated above $p$: its decomposition as a product of prime ideals takes the form $\Ja(\ba) \cdot \Z[\zeta_d] =  \prod_{i=1}^s\gp_i^{\ord_{\gp_i}\Ja(\ba)} $. 
Note that, for all $i\in\lint 1,s \rint$,
\[\ord_{\gp_i}\Ja(\ba)
= \ord_{\sigma_{g_i}^{-1}(\gp)} \Ja(\ba)
= \ord_{\gp}( \sigma_{g_i}\Ja(\ba))
= \ord_{\gp} \Ja(g_i\cdot\ba),
\]
where the last equality follows from \eqref{eq.actiongal.jacobi}.
Let us now assume that  $\Ja(\ba)\neq q^{|\bA|}$ and set
 \[I_\ba = \prod_{i=1}^s \gp_i^{\min\{v_\bA, \ord_{\gp}\Ja(g_i\cdot \ba)\}}
= \prod_{g\in(\Z/d\Z)^\times/\sbgp{p}{d}} (\sigma_g^{-1}\gp)^{\min\{v_\bA, \ord_{\gp}\Ja(g\cdot \ba)\}}.
\]
By construction, $I_\ba$ is an integral ideal that divides the nonzero ideal generated by $(q^{|\bA|}-\Ja(\ba))$ in $\Z[\zeta_d]$. 
In particular, $\norm I_\ba \leq \norm_{K/\Q}(q^{|\bA|}-\Ja(\ba))$ in $\Z$, which yields 
\[\norm_{K/\Q}\left(1- \frac{\Ja(\ba)}{q^{|\bA|}}\right)
= \frac{\norm_{K/\Q}(q^{|\bA|}-\Ja(\ba))}{\norm_{K/\Q}(q^{|\bA|})}
\geq \frac{\norm I_\ba}{\norm_{K/\Q}( q^{|\bA|})}
 = \frac{1}{q^{|\bA|\cdot\phi(d)} \cdot (\norm I_\ba )^{-1}}. \]
A straightforward computation shows that 
\[\log\norm_{K/\Q}\left(1- \frac{\Ja(\ba)}{q^{|\bA|}}\right)
\geq - \log(q^{|\bA|\cdot\phi(d)} \cdot (\norm I_\ba )^{-1})
= -\log q^{|\bA|} \cdot o_p(d) \cdot \sum_{i=1}^s \max\left\{ 0, 1- \frac{\ord_\gp \Ja(g_i\cdot \ba)}{v_\bA}\right\}. \]
Finally, we use our choice of $g_i$ as representatives of $(\Z/d\Z)^\times/\sbgp{p}{d}$ and the fact that $\Ja(p^j \cdot \ba)=\Ja(\ba)$ for all $j\geq 0$ to arrive at the expression announced in the statement of the lemma:
\begin{align*}
o_p(d) \cdot \sum_{i=1}^s \max\left\{ 0, 1- \frac{\ord_\gp \Ja(g_i\cdot \ba)}{v_\bA}\right\}
&= |\sbgp{p}{d}|\cdot \sum_{g\in(\Z/d\Z)^\times/\sbgp{p}{d}}\max\left\{ 0, 1- \frac{\ord_\gp \Ja(g\cdot \ba)}{v_\bA}\right\}\\
&= \sum_{g\in(\Z/d\Z)^\times}\max\left\{ 0, 1- \frac{\ord_\gp \Ja(g\cdot \ba)}{v_\bA}\right\}.
\end{align*}
\ProofEnd \end{proof}

We also need a more explicit expression of the $\gp$-adic valuations $\ord_\gp\Ja(t\cdot \ba)$ appearing in the previous lemma: 
\begin{lemm}[Stickelberger]\label{lemm.stickelberger}
 Let $\gp$ be as above. 
For any $\bb=(b_0, \dots,b_3)\in G_d^\circ\subset G_d$, set $v_\bB=\ord_p (q^{|\bB|}) = [\F_{q^{|\bB|}}: \F_p]$. The $\gp$-adic valuation of the Jacobi sum $\Ja(\bb)$ is given by
\begin{equation}\label{eq.stickelberger}
\frac{\ord_{\gp} \Ja(\bb) }{v_{\bB}}
= \frac{1}{|\sbgp{p}{d}|} \cdot  \sum_{\pi\in\sbgp{p}{d}}\left(3 -\sum_{i=0}^{3} \partfrac{\frac{b_i \cdot\pi}{d}}\right).
\end{equation}
\end{lemm}

\begin{proof}
For the duration of the proof, we set $Q=q^{o_q(d)}$, $v= [\F_Q:\F_p] = \ord_p Q$
and $q'=q^{|\bB|}$. 
The computation of the $\gp$-adic valuations of Gauss sums leading to Stickelberger's theorem  (as in \cite[Chap. 14]{IR} or \cite[IV, §3]{LANT}) 
implies that the Jacobi sum $\Ja_Q(\bb)$ (relative to $\F_Q$) has $\gp$-adic valuation:
\[ \ord_\gp\Ja_Q(\bb) = \sum_{j=0}^{v-1} \left( -1+\sum_{i=0}^3\partfrac{\frac{-b_i \cdot p^j}{d}} \right). \]
 Using the expression above, the relation of Davenport-Hasse (to the effect that $\Ja_Q(\bb) = \big(\Ja_{q'}(\bb) \big)^{[\F_Q:\F_{q'}]}$, see \eqref{eq.hassedav.jacobi}) and properties of $\ord_\gp$, this yields that:
 \begin{equation}\label{eq.stick.raw}
\ord_\gp \Ja(\bb) 
= \ord_\gp \Ja_{q'}(\bb) 
= \frac{1}{[\F_Q:\F_{q'}]} \ord_{\gp} \Ja_Q(\bb)
= \frac{1}{[\F_Q:\F_{q'}]} \sum_{j=0}^{v-1} \left( -1+\sum_{i=0}^3\partfrac{\frac{-b_i \cdot p^j}{d}} \right).\end{equation}

We now rearrange this raw expression. First of all, 
there may be repetitions in the sum over $j$:
since $v=\mathrm{lcm}([\F_q:\F_p], o_p(d))$, $v$ is a multiple of $o_p(d)$. 
By construction, $d$ divides $p^{o_p(d)}-1$ and any multiple thereof: it follows that 
we may reindex the sum over $j\in\intent{0,v-1}$ into a sum over $\pi\in\sbgp{p}{d}$ and obtain
\begin{equation}\label{eq.stick1}
\sum_{j=0}^{v-1} \left( -1+\sum_{i=0}^3\partfrac{\frac{-b_i \cdot p^j}{d}} \right)
= \frac{v}{o_p(d)} \cdot \sum_{\pi\in\sbgp{p}{d}} \left( -1+\sum_{i=0}^3\partfrac{\frac{-b_i \cdot \pi}{d}} \right). \end{equation}
Secondly, we note that $\partfrac{-y} = 1-\partfrac{y}$ for all $y\in\R\smallsetminus\Z$ and that $-b_i\cdot \pi /d$ is never an integer for $\pi\in\sbgp{p}{d}$ and $\bb \in G_d^\circ$. This leads to
\begin{equation}\label{eq.stick2}
\sum_{\pi\in\sbgp{p}{d}} \left( -1+\sum_{i=0}^3\partfrac{\frac{-b_i \cdot \pi}{d}} \right)
=\sum_{\pi\in\sbgp{p}{d}} \left( 3-\sum_{i=0}^3\partfrac{\frac{b_i \cdot \pi}{d}} \right).\end{equation}

To conclude, one only needs to combine \eqref{eq.stick1} and \eqref{eq.stick2}  with \eqref{eq.stick.raw}, and to remember that
\[\frac{v}{o_p(d) \cdot [\F_Q:\F_{q'}]}  = \frac{[\F_Q:\F_p]}{o_p(d) \cdot [\F_Q:\F_{q'}]} = \frac{[\F_{q'}:\F_p]}{o_p(d)} = \frac{v_\bB}{|\sbgp{p}{d}|}.\] \ProofEnd \end{proof}

We can now finish the proof of Theorem \ref{theo.proto.lowerbnd}:
as we noted at the beginning of this subsection, the desired lower bound on $P^\ast(\Lambda)$ follows from one on $\Pi^\ast_\Lambda$. 
The hypothesis that $\Lambda$ be stable under the action of $(\Z/d\Z)^\times$ implies the rationality of $\Pi^\ast_\Lambda$ (see \eqref{eq.actiongal.jacobi}). In particular, we get: 
\[\log |\spval{\Lambda}| 
\geq \log |\Pi^\ast_\Lambda|
= \frac{\log\norm_{K/\Q}(\Pi^\ast_\Lambda)}{[K:\Q]}
= \frac{1}{[K:\Q]}\sum_{\bA\in\O_q(\Lambda^\ast)} \log \norm_{K/\Q}\left( 1-\frac{\Ja(\ba)}{q^{|\bA|}} \right).
\]

We now use Lemma \ref{lemm.norm.jacobi} on each term of the sum:
\begin{align}
\log |\Pi^\ast_\Lambda|
&\geq \frac{1}{[K:\Q]}\sum_{\bA\in\O_q(\Lambda^\ast)} -\log(q^{|\bA|}) \cdot \sum_{g\in(\Z/d\Z)^\times} \max\left\{ 0, 1- \frac{\ord_{\gp} \Ja(g\cdot \ba)}{ v_{\bA} }\right\} \notag\\
&\geq \frac{-\log q}{\phi(d)} \sum_{g\in(\Z/d\Z)^\times}
\sum_{\bA\in\O_q(\Lambda^\ast)} |\bA| \cdot \max\left\{ 0, 1- \frac{\ord_{\gp} \Ja(g\cdot \ba)}{ v_{\bA} }\right\} 
\notag\\
&= \frac{-\log q}{\phi(d)} \sum_{g\in(\Z/d\Z)^\times}
\sum_{\ba\in\Lambda^\ast} \max\left\{ 0, 1- \frac{\ord_{\gp} \Ja(g\cdot \ba)}{ v_{\bA} }\right\}. \label{eq.prouto}
\end{align}
For all
 $\bb\in G_d^\circ\subset G_d$, 
 Lemma \ref{lemm.stickelberger} implies that
\begin{equation}\label{eq.express.ord}\notag
1- \frac{\ord_{\gp} \Ja(\bb)}{v_\bB}
= 1-  \frac{1}{|\sbgp{p}{d}|} \cdot  \sum_{\pi\in\sbgp{p}{d}}\left(3 -\sum_{i=0}^{3} \partfrac{\frac{b_i \pi}{d}}\right)
= \sum_{i=0}^{3} \left( -\frac{1}{2} + \frac{1}{|\sbgp{p}{d}|} \sum_{\pi\in\sbgp{p}{d}} \partfrac{\frac{b_i \pi}{d}} \right).\end{equation}
Plugging this  into \eqref{eq.prouto} with $\bb = g\cdot \ba\in\Lambda^\ast\subset G_d^\circ$ and exchanging the order of summation, we obtain
\[
\log |\Pi^\ast_\Lambda| 
\geq -(\log q)\cdot 
\sum_{\ba\in\Lambda^\ast} w_p(\ba,d) 
\geq -(\log q)\cdot 
\sum_{\ba\in\Lambda^\circ} w_p(\ba, d)  =- (\log q)\cdot w_p(\Lambda, d),
\]
because the terms $w_p(\ba,d)$ we added are nonnegative. This concludes the proof of Theorem \ref{theo.proto.lowerbnd}. \ProofEnd \end{proof}

\subsection{Proof of Theorem \ref{theo.lowerbnd}}


\newcommand{\good}[1]{{#1}^{\! \bullet}} 
\newcommand{\bad}[1]{{#1^{\dagger}}}

In Theorem \ref{theo.proto.lowerbnd}, we introduced the quantity $w_p(\ba, d)$: 
  \[ \forall \ba=(a_0, \dots, a_3)\in G_d^\circ, \quad 
  w_p(\ba, d)=  \frac{1}{\phi(d)} \sum_{g\in(\Z/d\Z)^\times}  \max\left\{ 0, 
 \sum_{i=0}^3\left(  - \frac{1}{2} + \frac{1}{|\sbgp{p}{d}|}\sum_{\pi\in\sbgp{p}{d}} 
\partfrac{\frac{a_i  \pi g}{d}} \right) 
 \right\}, \]
where $\sbgp{p}{d}\subset(\Z/d\Z)^\times$ denotes the subgroup generated by $p$. Notice that $w_p(\ba, d)$ satisfies a trivial upper bound:
\begin{equation}\label{eq.triv.upperbnd.w}
\forall \ba\in G_d^\circ, \qquad 0\leq w_p(\ba, d) \leq 1.
\end{equation}

In light of Theorem \ref{theo.proto.lowerbnd}, proving Theorem \ref{theo.lowerbnd}  
comes down to showing that, assuming \eqref{eq.hyp}, one has
\[ 
 w_p(\Lambda,d) 
= \sum_{\ba\in\Lambda^\circ} w_p(\ba, d) 
\ll_{p, \epsilon}   |\Lambda| \cdot \left(\frac{\log\log d}{\log d}\right)^{1/4-\epsilon}\]
That is, we have to show that  $w_p(\ba, d)$ is ``often small''. We start by showing that this is indeed the case for certain $\ba\in G_d^\circ$.

\begin{lemm} \label{lemm.good}
Let $d\geq 2$ be an integer, coprime with $q$. For an element $\ba=(a_0, \dots, a_3) \in G_d^\circ$, set $d_i = \gcd(d, a_i)$ and $d'=d/\max\{d_0, \dots, d_3\}$. Then, for all $\epsilon\in(0,1/4)$, one has 
\[ w_p(\ba, d)
\leq c_{10} \cdot \left(\frac{\log\log d'}{\log d'}\right)^{1/4-\epsilon},
\]
where the constant $c_{10}>0$ is effective and depends at most on $p$ and $\epsilon$.
\end{lemm}

\begin{proof}
Let $F:[0,1]\to\R$, $x\mapsto x$, and 
\begin{equation}\label{eq.def.theta}
\forall a\in \Z/d\Z \smallsetminus\{0\}, 
\qquad \theta_p(a, d):= \frac{1}{\phi(d)} \sum_{g\in(\Z/d\Z)^\times}  \left|
\int_{[0,1]} F - \frac{1}{|\sbgp{p}{d}|}\sum_{\pi\in\sbgp{p}{d}}  F\left(\partfrac{\frac{a \pi g}{d}}\right)    \right|.
\end{equation}
Observe that, for all $a\in\Z/d\Z\smallsetminus\{0\}$ and for all common divisors $\delta$ of $a$ and $d$, one has
$\theta_p(a,d) = \theta_p\left(a/\delta,d/\delta\right)$. 
Indeed, $\sbgp{p}{d/\delta}$ is the image of $\sbgp{p}{d}$ under the natural surjective morphism $(\Z/d\Z)^\times \to (\Z/(d/\delta) \Z)^\times$, which leads to
\[\frac{1}{|\sbgp{p}{d/\delta}|}\sum_{\pi'\in\sbgp{p}{d/\delta}}\partfrac{\frac{(a/\delta)\pi'}{d/\delta}} 
= \frac{1}{|\sbgp{p}{d}|}\sum_{\pi\in\sbgp{p}{d}}\partfrac{\frac{a\pi}{d}}, \]
and similarly for the outer average on $(\Z/d\Z)^\times$ in \eqref{eq.def.theta}.
In particular, if we set $a'=a/\gcd(d,a)$ and $d'=d/\gcd(d,a)$, then we have $\theta_p(a,d) = \theta_p\left(a',d'\right)$. 

Since $\gcd(a', d')=1$, we can use our equidistribution result (Theorem \ref{theo.equidis}) on $\theta_p(a',d')$, 
with 
$H_{d'} := \sbgp{p}{d'}$, seen as a subset of $(\Z/d'\Z)^\times$. 
Then, one has $|H_{d'}|\geq \log d'/\log p$ because, by definition of $o_p(d')=|\sbgp{p}{d'}|$, $d'$ divides $p^{|H_{d'}|}-1$. 
Therefore, applying Theorem \ref{theo.equidis} to this situation yields that, for all $a\in\Z/d\Z\smallsetminus\{0\}$, 
\begin{equation}\label{eq.majo.1}
\theta_p(a',d')\leq c_{7}\cdot(\log p)^{1/4-\epsilon}\cdot E(d')
\quad \text{where we have put } E(d'):=\left(\frac{\log\log d'}{\log d'}\right)^{1/4-\epsilon},\end{equation}
and where the constant $c_7$ depends at most on $\epsilon$.

For any $\ba\in G_d^\circ$ as in the statement of the Lemma, using the inequality $\max\{0,y\}\leq |y|$ ($y\in\R$)  and the triangle inequality, we note that,
\[w_p(\ba, d) 
\leq \frac{1}{\phi(d)} \sum_{g\in(\Z/d\Z)^\times}  \left| 
 \sum_{i=0}^3  \left(- \frac{1}{2} + \frac{1}{|\sbgp{p}{d}|}\sum_{\pi\in\sbgp{p}{d}} 
\partfrac{\frac{a_i  \pi g}{d}} \right)\right|
\leq \sum_{i=0}^3 \theta_p(a_i, d).\] 
Upon applying Theorem \ref{theo.equidis} to each $a_i\in\Z/d\Z\smallsetminus\{0\}$ as in the previous paragraph, we get
\[w_p(\ba, d)
\leq \sum_{i=0}^3 \theta_p(a_i, d) 
\leq c_{7} \cdot (\log p)^{1/4-\epsilon} \cdot \sum_{i=0}^{3} E(d/d_i)
\leq 4 c_{7} \cdot (\log p)^{1/4-\epsilon}\cdot E(d/\max_i\{d_i\}),
\]
because $n\mapsto E(n)$ is a decreasing function. Taking $c_{8} = 4  c_{7} \cdot (\log p)^{1/4-\epsilon}$ concludes the proof. 
\ProofEnd \end{proof}

 
In the case that $d$ is assumed to be prime, the proof of Theorem \ref{theo.lowerbnd} follows directly from Lemma \ref{lemm.good} without extra hypotheses on $\Lambda$.
 Indeed, for a prime $d$ and for any nonempty $(\Z/d\Z)^\times$-stable subset $\Lambda\subset G_d$, all the elements $\ba=(a_0, \dots, a_3)\in\Lambda\cap G_d^\circ$ have $\max_i\{\gcd(d,a_i)\}=1$. 
Thus, Lemma \ref{lemm.good} implies that, for all $\epsilon\in(0,1/4)$, one has
\[w_p(\Lambda, d) 
= \sum_{\ba\in\Lambda^\circ} w_p(\ba, d)
\leq  \sum_{\ba\in\Lambda^\circ} c_{8}   \cdot \left(\frac{\log\log d}{\log d}\right)^{1/4-\epsilon} 
\leq c_{8}   \cdot |\Lambda| \cdot  \left(\frac{\log\log d}{\log d}\right)^{1/4-\epsilon},
\]
as claimed. 
\begin{proof}[of Theorem 6.2] 
Let $d\geq 2$ be any integer prime to $q$, $\Lambda\subset G_d$ be a nonempty $(\Z/d\Z)^\times$-stable subset which satisfies hypothesis \eqref{eq.hyp}, and $\epsilon\in(0,1/4)$. For 
$X\in(1,d)$, define two subsets of~$\Lambda^\circ$:
\[\good{\Lambda}(X) = \left\{\ba=(a_0, \dots, a_3)\in\Lambda^\circ \ \big| \ \max_i\{\gcd(d,a_i)\}\leq X\right\} \]
and $\bad{\Lambda}(X) := \Lambda^\circ\smallsetminus\good{\Lambda}(X)=  \left\{\ba\in\Lambda^\circ \ \big| \  \max_i\{\gcd(d,a_i)\}> X\right\}$. 
We start by cutting the sum $w_p(\Lambda,d)$ into two parts:
\[ w_p(\Lambda, d) 
= \sum_{\ba \in \Lambda^\circ} w_p(\ba, d)
= \sum_{\ba \in \good{\Lambda}(X)} w_p(\ba, d) + \sum_{\ba \in \bad{\Lambda}(X)} w_p(\ba, d). \]
Since $w_p(\ba,d)$ satisfies the trivial bound \eqref{eq.triv.upperbnd.w}, the second sum is less than $|\bad{\Lambda}(X)|$. 
To bound the first sum, we make use of Lemma \ref{lemm.good} and of the bound $|\good{\Lambda}(X)|\leq |\Lambda|$. 
This yields:
\[  \sum_{\ba \in \good{\Lambda}(X)} w_p(\ba, d)
\leq \sum_{\ba \in \good{\Lambda}(X)} c_{8}   \cdot \left(\frac{\log\log (d/X)}{\log (d/X)}\right)^{1/4-\epsilon} 
\leq c_{8}   \cdot |\Lambda| \cdot  \left(\frac{\log\log (d/X)}{\log (d/X)}\right)^{1/4-\epsilon}.
\]
Summing these two contributions, we arrive at:
\[ w_p(\Lambda, d) 
= \sum_{\ba \in \Lambda^\circ} w_p(\ba, d)
\leq  |\Lambda|\cdot \left( c_{8}   \cdot  \left(\frac{\log\log (d/X)}{\log (d/X)}\right)^{1/4-\epsilon}  
+ \frac{|\bad{\Lambda}(X)|}{|\Lambda|}
\right). \]
By hypothesis \eqref{eq.hyp}, for the given $\epsilon\in(0,1/4)$, there exists $u\in (0,1)$ and a constant $c'$ such that 
${\bad{|\Lambda}(d^u)|}/{|\Lambda|}\leq c' \cdot \left(\frac{\log\log  d}{\log d}\right)^{1/4-\epsilon}$. So that, on choosing $X=d^u$, we obtain that
\[ \frac{w_p(\Lambda, d)}{|\Lambda|}
\leq \frac{c_8}{1-u} \cdot \left(\frac{\log\log  d}{\log d}\right)^{1/4-\epsilon}  + c' \cdot \left(\frac{\log\log  d}{\log d}\right)^{1/4-\epsilon}.\]
Setting $C_1 = c_8/(1-u)  + c'>0$, we have proved the claimed lower bound on $P^\ast(\Lambda)$. 
\ProofEnd\end{proof}

\subsection{Conclusion}

Let us recapitulate the results we have obtained so far: let $\F_q$ be a finite field of characteristic $p$, $d\geq 2$ be an integer prime to $q$, and denote by $\ferm_d/\F_q$ the Fermat surface of degree $d$ over $\F_q$. Then, in the notations of sections \ref{sec.zeta.ferm} and \ref{sec.setting}, one has
\[P_2(\ferm_d/\F_q, T) = P(G_d, T)\in\Z[T]\]
(see Theorem \ref{theo.weil} and Example \ref{ex.fermat.P2}). So that, on the one hand, one has 
\begin{equation}\label{eq.conclu1}
P^\ast(G_d) = P_2^\ast(\ferm_d/\F_q, q^{-1})= \frac{|\Br(\ferm_d)|\cdot \Reg(\ferm_d)}{q^{p_g(\ferm_d)}} \in \Z[q^{-1}]\smallsetminus\{0\}
\end{equation}
by the Artin-Tate Conjecture (Theorem \ref{theo.shioda.AT} and Proposition \ref{prop.fermat.spval}). On the other hand, provided that $G_d$ satisfies hypothesis \eqref{eq.hyp}, 
we have proved that
\begin{equation}\label{eq.conclu2}
\forall \epsilon\in(0,1/4),\quad 
 - C_1 \cdot\left(\frac{\log\log d}{\log d}\right)^{1/4-\epsilon}\leq \frac{\log |P^\ast(G_d)|}{  \log q^{|G_d|} } 
 \leq C_2 \cdot\frac{\log\log |G_d|}{\log|G_d|},
 \end{equation}
for some constants $C_1, C_2$ depending on $q$, $p$ and $\epsilon$ (see Theorems \ref{theo.upperbnd} and \ref{theo.lowerbnd}).
Since $|G_d| = d^3$ and $p_g(\ferm_d)\sim d^3/6$ (as $d\to\infty$), the identity \eqref{eq.conclu1} and the inequalities \eqref{eq.conclu2} directly imply our main result (Theorem \ref{theo.BS.FERMAT}). 
Consequently, there only remains to prove that $G_d$  indeed satisfies \eqref{eq.hyp}  for any integer $d\geq 2$  coprime with $q$.

\begin{lemm}\label{lemm.bad}
Let $d\geq 2$ be an integer and, for $X\in(1,d)$, define the subset 
\[\bad{G_d}(X):=\left\{\ba=(a_0, \dots, a_3)\in G_d \ \big| \ d>\max_i\{\gcd(d,a_i)\} > X\right\}\subset G_d^\circ.\]
Then, for all $u\in(0,1)$, there exists a constant $c_9$ depending only on $u$ such that: 
\begin{equation}\label{eq.lemmabad}
|\bad{G_d}(d^u)|
\leq c_9 \cdot d^{-u/2}\cdot |G_d|.
\end{equation}
In particular, $G_d$ satisfies hypothesis \eqref{eq.hyp}.
\end{lemm} 

\begin{proof} For any divisor $g$ of $d$ and $i\in\intent{0,3}$, let $H_{i}(g)$ be the subgroup of $G_d$ defined by \[H_i(g) = \left\{ \ba=(a_0, \dots, a_3)\in G_d \ \big| \ a_i\equiv 0 \bmod{g}\right\}.\]
Since the group homomorphism $G_d \to \Z/g\Z$ given  by $\ba \mapsto a_i\bmod{g}$ is surjective with kernel $H_i(g)$, the isomorphim theorem yields $|H_i(g)|=|G_d|/g = d^3/g$.
On the other hand, one has 
\begin{align*}
|\bad{G_d}(X) |
&= \sum_{\substack{g\mid d \\ X<g<d}} \left| \left\{ \ba\in G_d \big| \max_i\{\gcd(d, a_i)\} = g\right\}\right|
\leq \sum_{\substack{g\mid d \\ X<g<d}} \sum_{i=0}^3 \left|\left\{ \ba\in G_d \big| \gcd(d, a_i) = g\right\}\right| \\
&\leq \sum_{\substack{g\mid d \\ X<g<d}} \sum_{i=0}^3 |H_i(g)| = \sum_{i=0}^3\sum_{\substack{g\mid d \\ X<g<d}} \frac{|G_d|}{g} 
\leq 4\cdot \frac{|G_d|}{X} \cdot \left| \left\{ g \in\intent{X, d} \text{ such that } g\mid d 
\right\}\right|.
\end{align*}
Denoting by $\tau(d)$ the number of divisors of $d$, the set $ \left\{ g \in\intent{X, d} \text{ such that } g\mid d \right\}$ has less than $\tau(d)$ elements.
It is classical that, for all $v>0$, there is an explicit constant $c_v$ such that $\tau(d) \leq c_v \cdot d^v$  (see \cite[Thm. 315]{HardyWright}). 
In particular, for $v=u/2 >0$, we have proved:
\[ |G_d(d^u)| 
\leq 4 \cdot \frac{|G_d|}{d^u}  \cdot \tau(d) 
\leq 4 c_{u/2}\cdot d^{-u/2} \cdot |G_d|,\]
which is the claimed upper bound \eqref{eq.lemmabad}.  
Hypothesis  \eqref{eq.hyp} for $G_d$ is then a  consequence of the fact that, for all $\epsilon \in(0,1/4)$ and all $u\in(0,1)$, one has  $d^{-u/2} = o\left(\left(\tfrac{\log\log d}{\log d}\right)^{1/4-\epsilon}\right)$ as $d\to\infty$.
%
\ProofEnd\end{proof}

Finally, putting \eqref{eq.conclu1} and \eqref{eq.conclu2} together, we obtain a strong version of Theorem \ref{theo.BS.FERMAT}:
\begin{coro}\label{theo.BNDS.GEN} Let $\F_q$ be  a finite field of characteristic $p$, and $\epsilon \in (0,1/4)$. For an integer $d\geq 2$ prime to $q$, as $d\to\infty$, 
one has 
\begin{equation}\label{eq.SPVALBOUND.goal}
\frac{\log\left( |\Br(\ferm_d)|\cdot \Reg(\ferm_d)\right)}{\log q^{p_g(\ferm_d)}} 
= 1+ \frac{\log P^\ast(G_d)}{\log q^{p_g(\ferm_d)}} 
= 1 + O\left(\left(\tfrac{\log\log d}{\log d}\right)^{1/4-\epsilon}\right),
\end{equation}
where the implicit constants 
are effective, and depend at most on $q$, $p$ and $\epsilon$.
\end{coro}


%
%
%



\newcommand{\mapolicebackref}[1]{%
         \hspace*{-5pt}{\textcolor{gray}{\footnotesize$\uparrow$ #1}}
}
\renewcommand*{\backref}[1]{
\mapolicebackref{#1}
}
\hypersetup{linkcolor=gray}

\pdfbookmark[0]{References}{references} 
\addcontentsline{toc}{section}{References}
\bibliographystyle{alpha}
\bibliography{Biblio_Art1} 


\end{document}